\documentclass[11pt,oneside,a4paper,natbib=false]{article}
\usepackage{geometry}
\usepackage{ifxetex,ifluatex}
\newif\ifxetexorluatex
\ifxetex
  \xetexorluatextrue
\else
  \ifluatex
    \xetexorluatextrue
  \else
    \xetexorluatexfalse
  \fi
\fi

\usepackage[final]{microtype}

\ifxetexorluatex%
	\usepackage{amsthm,amsmath,amssymb}
	\usepackage{fontspec}

	\usepackage{csquotes}
	\usepackage{polyglossia}
    \setdefaultlanguage{english}

	\makeatletter %
		\newcommand{\blx@nowarnpolyglossia}{}
	\makeatother
\else%
	\RequirePackage[utf8]{inputenc}

	\RequirePackage{amsfonts}
	\RequirePackage{amsthm,amsmath}
	\RequirePackage{amssymb}

	\usepackage{csquotes}
	\usepackage[english]{babel}
\fi

\usepackage[backend=biber, style=numeric, sorting=anyt, autopunct=true,
maxnames=10, maxalphanames=10, natbib=true, arxiv=abs, doi=true, url=false,
eprint=true]{biblatex}
\usepackage[title,titletoc]{appendix}

\usepackage{graphicx}

\usepackage{float}
\floatstyle{plain}
\restylefloat{figure}

\usepackage{subcaption}
\usepackage[ruled,vlined,linesnumbered]{algorithm2e}

\usepackage{float}
\usepackage{enumitem}
\usepackage[usenames,dvipsnames,svgnames,table]{xcolor}

\usepackage{booktabs}
\usepackage{multirow}
\usepackage{slashed}

\usepackage{hyperref}

\hypersetup{%
	unicode=true,
	pdftitle={Covering simple orthogonal polygons with r-stars},
	pdfauthor={Tamás Róbert Mezei},
	pdfsubject={Mathematics, Computer Science},
	pdfkeywords={r-star; r-visibility; simple orthogonal polygons;
    polygon cover; art gallery;},
	colorlinks=false,
}

\usepackage{cleveref}
\Crefname{assumption}{Assumption}{Assumptions}

\usepackage{tikz}
\usepackage{pgfplots}
\pgfplotsset{compat=newest}
\tikzset{every picture/.style={line cap=butt,thick}}

\newtheorem{theorem}{Theorem}
\newtheorem{lemma}[theorem]{Lemma}
\newtheorem{corollary}[theorem]{Corollary}
\newtheorem{conjecture}[theorem]{Conjecture}
\newtheorem{observation}[theorem]{Observation}
\theoremstyle{definition}
\newtheorem{definition}[theorem]{Definition}
\newtheorem{assumption}[theorem]{Assumption}

\newtheorem{remark}[theorem]{Remark}

\DeclareMathOperator{\NCA}{NCA}
\DeclareMathOperator{\interior}{int}
\DeclareMathOperator{\closure}{cl}
\newcommand{\SymmDiff}{\operatornamewithlimits{\scalebox{1.3}{$\Delta$}}}

\title{Covering simple orthogonal polygons with $r$-stars}

\vspace{2ex}

\author{%
	Tamás Róbert Mezei\footnote{Supported in part by the National Research,
    Development and Innovation Office (NKFIH) grants SNN-135643 and K-132696.} \\
	\texttt{mezei.tamas.robert@renyi.hu}
}
\date{%
	\small Alfréd Rényi Institute of Mathematics \\ Reáltanoda~u.~13--15, 1053 Budapest, Hungary \\
	\vspace{12pt}
	\normalsize \today
	\vspace{-12pt}
}
\vspace{2ex}

\begin{document}

\maketitle

\centerline{\emph{Dedicated to the memory of János Urbán.}}

\begin{abstract}
	We solve the $r$-star covering problem in simple orthogonal polygons, also
	known as the point guard problem in simple orthogonal polygons with
	rectangular vision, in quadratic time.

\end{abstract}

\section{Introduction}

Art gallery problems in general ask about the minimum number of guards with
given power \emph{(for example, static or mobile)} and type of vision
\emph{(line of sight, rectangular vision, etc.)} required to control the
gallery. A \emph{point guard} is a point in the interior of a polygon, and it
covers any point in the closed polygon (\emph{the gallery}) to which the guard
can be joined by a line segment contained by the closed polygon (\emph{line of
sight vision}). The art gallery theorem due to \citet{chvatal_combinatorial_1975} states that
given an $n$-vertex simple polygon, $\lfloor\frac{n}{3}\rfloor$ point guards are
sufficient and sometimes necessary to cover the closed polygon. In 1980 the
sharp bound for the special case of $n$-vertex simple orthogonal polygons was
determined to be $\lfloor \frac{n}{4}\rfloor$ by \citet{kahn_traditional_1983}.

\medskip

However, as \citet{schuchardt_two_1995} proved, determining the exact number
of point guards is \textsc{NP}-hard even in simple orthogonal polygons. To have a chance at
computing a minimum cardinality set of covering point guards in polynomial time, line of
sight vision needs to be restricted.

\medskip

Two points in an orthogonal polygon (using axis-parallel sides) have
\emph{rectangular vision} or \emph{$r$-vision} of each other if there is a
non-degenerate axis-parallel rectangle that contains both points and the rectangle is
contained in the region bounded by the polygon. An $r$-star is a region
that contains a point from which every point of the region is
$r$-visible.  To keep the terminology concise, from now on, a point guard equipped
with $r$-vision is referred to as an \emph{$r$-guard}.

\medskip

\citet{gyori_short_1986,orourke_alternate_1983} independently proved that there
is a stronger combinatorial theorem behind the orthogonal art gallery theorem of
\citet{kahn_traditional_1983}: any $n$-vertex simple orthogonal polygon can be
partitioned into at most $\lfloor\frac{n}{4}\rfloor$ simple orthogonal polygons
of at most 6 vertices. Notice, that this immediately implies that the extremal
result of \citeauthor{kahn_traditional_1983} holds even for $r$-guards. In other
words, $r$-vision is a reasonable restriction of line of sight vision in the
case of orthogonal polygons, because it performs equally in the extremal case.

\medskip

The seminal result of \citet{worman_polygon_2007} shows that finding the minimum
cardinality set of $r$-guards that cover a simple orthogonal polygon is
more feasible than using guards with line of sight vision.
\begin{theorem}[\citet{worman_polygon_2007}]
	A minimum cardinality $r$-star cover of a simple orthogonal polygon can be found in
	${\mathcal{O}}(n^{17}\cdot\mathrm{poly}(\log n))$ time.
\end{theorem}
The algorithm of \citeauthor{worman_polygon_2007} is still quite slow in
practice. Before we describe approximation results and faster algorithms for
special cases of the problem, let us review what is considered
a solution to the $r$-star cover problem in the literature.
\begin{enumerate}[label={(\Alph*)}]
	\item\label{type:A} The minimum number of $r$-stars that cover the whole region, and a
		list of the vertices of each $r$-star shaped polygon in the cover.
		This is the canonical meaning, and we will refer to it simply as the
		$r$-star cover problem.
    \item\label{type:B} A minimum cardinality set of $r$-guards, such that any
        point in the whole region is $r$-visible to at least one of the listed
        guards. We will refer to this variant as the $r$-guard cover problem.
    \item\label{type:C} Counting the minimum number of $r$-guards necessary to
        cover the region. This is equivalent to counting the minimum number of
        $r$-stars necessary to cover the region.
\end{enumerate}
\citet{culberson_orthogonally_1989} showed, that there exists a family of simple
orthogonal polygons, such that any type~\ref{type:A} solution where the
$r$-stars are maximal has space complexity $\Omega(n^2)$ in the size of the
output for an infinitely large subset of inputs. However, no such lower bound is
known for type~\ref{type:B} and~\ref{type:C} solutions.

A class-$k$ orthogonal polygon (where $0\le k\le 4$) is a simple orthogonal
polygon that has dents along at most $k$ directions out of the 4 possible
axis-parallel directions. In class-2a polygons, these two directions are
parallel, and in class-2b polygons the directions of the dents are orthogonal.
In the type~\ref{type:A} sense, the following algorithm is asymptotically
optimal.
\begin{theorem}[\citet{keil_minimally_1986}]
	The $r$-star cover problem can be solved in $\mathcal{O}(n^2)$ time in
	class-2a orthogonal polygons.
\end{theorem}
The result has been extended in two different directions.
\begin{theorem}[\citet{culberson_orthogonally_1989}]
	The $r$-star cover problem can be solved in $\mathcal{O}(n^2)$ time in
	class-2b orthogonal polygons. The count of the minimum number of $r$-guards necessary to
	cover a class-2 orthogonal polygon can be computed in $\mathcal{O}(n)$ time.
\end{theorem}
For class-3 polygons, an output-sensitive log-linear algorithm exists.
\begin{theorem}[\citet{palios_minimum_2015}]
    A minimum cardinality set of $r$-guards necessary to
	cover a class-3 orthogonal polygon can be computed in $\mathcal{O}(n+k\log
    k)$ time, where $k$ is the cardinality of the output solution.
\end{theorem}

In the realm of approximation algorithms, even faster algorithms have been
discovered.\@ \citet{lingas_linear-time_2012} gave a linear time approximation algorithm for
the $r$-guarding problem. Their proof partitions simple orthogonal polygons into staircase
shaped regions, where the problem can be solved exactly in linear time in the
type~\ref{type:B} sense.
\begin{theorem}[\citet{lingas_linear-time_2012}]\label{thm:lingas}
	A $3$-approximation solution to the $r$-guard cover problem can be computed
    in $\mathcal{O}(n)$ time.
\end{theorem}
A slightly better approximation of the minimum number of $r$-guards can be computed
similarly efficiently.
\begin{theorem}[\citet{gyori_mobile_2019}]\label{thm:vs}
    An $\frac83$-approximation of the count of the minimum number of $r$-guards
    required to cover $P$ can be computed in $\mathcal{O}(n)$ time.\footnote{The author of this paper
	has claimed that the algorithm can actually provide a set of $r$-guards in
    linear time, but only a sketch of the proof was presented at
    ICGT~2018~\cite{mezei_linear_2018}. The complete proof was never published.}
\end{theorem}
Guards equipped with $r$-vision that can move along an axis-parallel line
segment in the gallery are called sliding cameras by~\citet{katz_guarding_2011}
(in general polygons, guards that can move along a line segment are called
mobile guards). \Cref{thm:vs} is based on an extremal result that established an
inequality between the minimum number of $r$-guards and the minimum number of
horizontal and vertical sliding cameras required to control the gallery. The
main computational result behind \Cref{thm:vs} is the following theorem.
\begin{theorem}[\citet{gyori_mobile_2019}]\label{thm:linear_mhsc}
	The {minimum cardinality horizontal sliding cameras (MHSC)} problem can be solved in
    $\mathcal{O}(n)$ time.
\end{theorem}

\paragraph{Contribution.}
The objective of this paper is to prove the following result.
\begin{theorem}\label{thm:alg_typeA}
	A minimum cardinality set of covering $r$-stars of a simple orthogonal polygon can be
    computed in $\mathcal{O}(n^2)$ time.
\end{theorem}
\Cref{thm:alg_typeA} is asymptotically optimal in the sense that there exist
problems for which the description of every minimum cardinality cover by
\emph{maximal} $r$-stars that consists
of the vertex lists of the $r$-stars has size
$\Omega(n^2)$, see~\citet{culberson_orthogonally_1989}. The proof of
\Cref{thm:alg_typeA} entails three parts, namely \Cref{thm:Wk_covers_G},
\Cref{thm:independent}, and \Cref{thm:running_time}. The main idea behind the
proof of \Cref{thm:alg_typeA} is inspired by the algorithm of 
\Cref{thm:linear_mhsc}.

\medskip

Like prior work on orthogonal
polygons~\cite{motwani_covering_1988,culberson_orthogonally_1989,worman_polygon_2007,katz_guarding_2011},
our approach translates the $r$-star cover problem into a more general
(extremal) graph theory problem. The distinguishing feature of our approach is
that regions of the gallery are mapped to edges instead of vertices; visibility
between edges will hold if both edges are present in a homomorphic image of
$C_4$. In this graph, a minimum cardinality of covering $r$-guards will be
computed by \Cref{alg:main}. The optimality of the solution provided by
\Cref{alg:main} will be shown using the min-max equality described in
\Cref{thm:independent}. In the following sections, we argue and demonstrate that
\Cref{thm:alg_typeA} holds.

\section{Preliminaries and tree-based bipartite graphs}\label{sec:prelim}
Let $P$ be a simple orthogonal polygon, that is, a polygon that does not
intersect itself and whose sides are axis-parallel. The region bounded by
such a polygon does not have holes.
Let $G_H$ be the set of internally disjoint rectangles obtained by cutting
horizontally at each reflex vertex of $P$; we call these the \emph{horizontal
slices}. The horizontal slices cover the region bounded by $P$.
Similarly, let the \emph{vertical slices} $G_V$ be
defined analogously for vertical cuts of $P$. The vertical slices also cover
$P$. Note that $|G_H|+|G_V|$ is at most the number of vertices of $P$, because
every reflex vertex of $P$ is incident on one horizontal and one
vertical cut. Furthermore, every cut is incident on at most two reflex vertices,
therefore the number of vertices of $P$ is at most $2(|G_H|+|G_V|)+4$.
We will describe the complexity of the problems as a function of
$n=|G_H|+|G_V|$.
\begin{definition}\label{def:rastergraph}
Let $G$ be the intersection graph of $G_H$ and $G_V$, i.e.,
\begin{equation*}
G=\left(G_H,G_V;\left\{\{h,v\} \Big.\mid h\in G_H,\ v\in G_V,\
\interior(h)\cap \interior(v)\neq\emptyset\right\}\right).
\end{equation*}
In other words, a horizontal and a vertical slice are joined by an edge if and
only if their interiors intersect. We refer to $G$ as the \emph{raster
graph} of $P$.
\end{definition}

Note, that $G_H$ and $G_V$ are disjoint sets,
unless $|G_H|=|G_V|=1$; since in this case \Cref{thm:alg_typeA}
is trivial, from now on, let us assume that $G_H\cap G_V=\emptyset$. Because of
this, we may define
\begin{equation*}
    hv=vh=\{h,v\}\qquad \forall h\in G_H\ \forall v\in G_V
\end{equation*}
without risk of ambiguity. When possible, we will refer to edges in row-column
order $hv$, because it mimics how matrices are usually indexed. However, for some
$st\in E(G)$ where it is not known a priori whether $s\in G_H$ or $s\in G_V$, adherence
to the previous convention cannot be guaranteed.

The \emph{pixel} of an edge $hv\in E(G)$ is $h\cap v$. The \emph{set of pixels}
\begin{equation*}
\mathrm{pixels}(P)=\big\{\, h\cap v \mid hv\in E(G)\big\}
\end{equation*}
is obviously a cover of the region bounded by $P$.
\begin{definition}
    Let $\Gamma(s)$ --- absent a subscript --- denote the set of neighbors
    joined to a vertex $s\in G_H\cup G_V$ in the graph $G$.
\end{definition}

\begin{definition}\label{def:rvisionedge}
    An edge $hv\in E(G)$ is $r$-visible to $h'v'\in
    E(G)$ if and only if both $hv',h'v\in E(G)$. We denote this symmetric
    relation by $hv\boxdot h'v'$.
\end{definition}
\Cref{def:rvisionedge} is demonstrated on \Cref{fig:rvisionedge}.
The next lemma translates $r$-vision in a simple orthogonal polygon to $r$-vision
in its raster graph.

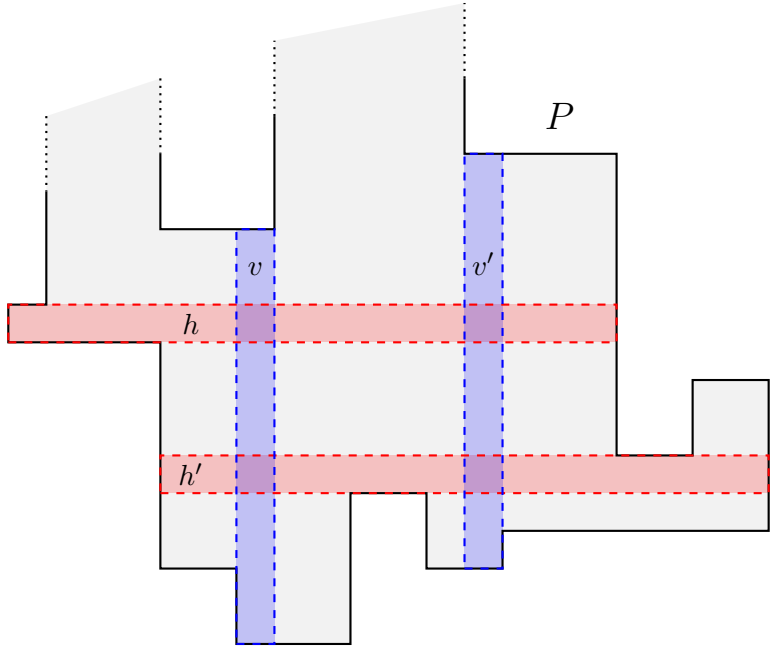
\begin{figure}[H]
	\centering
	\begin{tikzpicture}
        \fill[gray,opacity=0.1] (-1.5,5)--(-1.5,2.5)--(-2,2.5)--(-2,2)--(0,2)--(0,-1)--(1,-1)--(1,-2)--(2.5,-2)--(2.5,0)--(3.5,0)--(3.5,-1)--(4.5,-1)
            --(4.5,-0.5)--(8,-0.5)--(8,1.5)--(7,1.5)--(7,0.5)--(6,0.5)--(6,4.5)--(4,4.5)--(4,6.5)
            -- (1.5,6)--(1.5,3.5)--(0,3.5)--(0,5.5);

        \fill[red, fill opacity=0.2] (0,0) rectangle (8,0.5);
        \fill[red, fill opacity=0.2] (-2,2) rectangle (6,2.5);
        \fill[blue,fill opacity=0.2] (1,-2) rectangle (1.5,3.5);
        \fill[blue,fill opacity=0.2] (4,-1) rectangle (4.5,4.5);

        \draw[dotted] (-1.5,5)--(-1.5,4);
        \draw[dotted] (0,5.5)--(0,4.5);
        \draw[dotted] (1.5,5)--(1.5,6);
        \draw (1.5,5)--(1.5,3.5)--(0,3.5)--(0,4.5);
        \draw (-1.5,4)--(-1.5,2.5)--(-2,2.5)--(-2,2)--(0,2)--(0,-1)--(1,-1)--(1,-2)--(2.5,-2)--(2.5,0)--(3.5,0)--(3.5,-1)--(4.5,-1)
            --(4.5,-0.5)--(8,-0.5)--(8,1.5)--(7,1.5)--(7,0.5)--(6,0.5)--(6,4.5)--(4,4.5)--(4,5.5)
            edge[dotted] (4,6.5);
        \draw[thick,dashed,red] (0,0) rectangle (8,0.5);
        \draw[thick,dashed,red] (-2,2) rectangle (6,2.5);
        \draw[thick,dashed,blue] (1,-2) rectangle (1.5,3.5);
        \draw[thick,dashed,blue] (4,-1) rectangle (4.5,4.5);

        \node[anchor=center] at (5.25,5) {\Large $P$};
        \node[anchor=center] at (0.4,0.25) {$h'$};
        \node[anchor=center] at (0.4,2.25) {$h$};
        \node[anchor=south] at (1.25,2.75) {$v$};
        \node[anchor=south] at (4.25,2.75) {$v'$};
	\end{tikzpicture}
    \caption{The simple orthogonal polygon $P$ is drawn with solid segments, and
        the region bounded by it is filled with gray.
        The horizontal $h,h'\in G_H$ and vertical slices $v,v'\in
    G_V$ are filled with red and blue, respectively, while their boundaries are
    contoured with dashed lines. Note, that $hv,hv',h'v,h'v'\in E(G)$, thus $hv\boxdot h'v'$ holds.}\label{fig:rvisionedge}
\end{figure}

\begin{lemma}\label{lemma:rvisionedge}
    Let $x\in \mathbb{R}^2$ be arbitrary. Then
    \begin{equation*}\label{eq:rvisionedge}
            \{ y\in \mathbb{R}^2 \mid y\text{\ $r$-visible to $x$ in $P$}\}
            =\bigcup_{\substack{x\in h\cap v \\ hv\in
            E(G)}}\bigcup_{\substack{hv\boxdot h'v' \\ h'v'\in E(G)}} h'\cap v'.
    \end{equation*}
\end{lemma}
\begin{proof}
    We prove the equation by showing that the two sides contain each other.
    \begin{itemize}
        \item Let $hv,h'v'\in E(G)$ such that $hv\boxdot h'v'$, and suppose that
            $x\in h\cap v$ and $y\in h'\cap v'$. By
            $r$-visibility, we know that $hv',h'v\in E(G)$. Note, that the
            convex hull of $h\cap v$ and $h\cap v'$ is contained in $h$, since
            $h$ is a rectangle. Similar statements hold for $v,h',v'$. Since $P$
            is simple, the convex hull of $(h\cap v)\cup (h\cap v')\cup (h'\cap
            v)\cup (h'\cap v')$ is contained in the region bounded by $P$, therefore $y$ is
            $r$-visible to $x$.
    
        \item Suppose $y$ is $r$-visible to $x$ in $P$. Let $R$ be a minimum
            size axis-parallel rectangle that contains both $x$ and $y$. Extend
            the pairs of parallel sides of $R$ until all sides intersect the polygon
            $P$; let the extended rectangle be $R'$. Then $R'$ is the union of
            pixels. Take any $hv,h'v'\in E(G)$ such that $x\in h\cap v$, $y\in
            h'\cap v'$, and both $h\cap v,h'\cap v'\subseteq R'$. Because $R'$
            is a rectangle that intersects the interiors of $h,h',v,v'$ each, we
            have $\interior(h)\cap \interior(v')\neq\emptyset$ and
            $\interior(h')\cap \interior(v)\neq\emptyset$.
    \end{itemize}
\end{proof}
\begin{lemma}\label{lemma:restriction_rguard}
    Requiring that $r$-guards be placed in the interior of pixels does not
    reduce the generality of the $r$-guarding problem. The general problem can
    be reduced to the restricted version where guards are placed in the interior
    of pixels in $\mathcal{O}(n)$.
\end{lemma}
\begin{proof}[Proof.]
    Let $\delta\in \mathbb{R}^+$ be the minimum non-zero difference between the $x$-coordinates or $y$-coordinates
    of two vertices of $P$.
    Let $S$ be an axis-parallel square centered on the origin with sides of
    length $2\varepsilon$, where $\varepsilon=\frac{1}{4}\delta$. Let $Q$ be
    the simple orthogonal polygon bounding the Minkowski-sum of $S$ and the
    region bounded by $P$. The orientation of the corner (which way the right
    angle faces) at a vertex $v$ of $P$ determines whether $\pm\varepsilon$ has
    to be added to the coordinates of $v$ to obtain its image in $Q$. The choice
    of $\varepsilon$ guarantees that $Q$ is not self-intersecting and that every
    pixel of $Q$ has a minimum side length of $2\varepsilon$. Furthermore, if a
    point in the region bounded by $Q$ is at least $\varepsilon$ distance from
    $Q$, then it must be in the closed region bounded by $P$.

    By \Cref{lemma:rvisionedge} there exists
    a minimum cardinality set $X$ of $r$-guards of $P$ such that every $r$-guard is
    located at the corner of some pixel of $P$. Observe, that an $x\in X$ is not
    contained by the boundary of any pixel of $Q$. Note
    also, that the Minkowski-sum of $S$ and a rectangle is still a rectangle,
    therefore $X$ also covers $Q$.

    \medskip

    In the other direction, let $Y$ be a minimum cardinality covering set of $r$-guards of $Q$ that
    avoid the boundaries of the pixels. By \Cref{lemma:rvisionedge}, we may
    choose $Y$ such that every $r$-guard $y\in Y$ is
    at the center of its enveloping pixel in $Q$. First, observe that every
    $y\in Y$ is also contained in the region bounded by $P$, because the center
    of a pixel of $Q$ is at least $\varepsilon$ distance away from $Q$.
    Similarly, the center $x$ of a pixel of $P$ is at least $\varepsilon$ distance
    away from the boundary of the pixel of $Q$ containing $x$. By
    \Cref{lemma:rvisionedge}, if $x$ is $r$-visible to $y$ in $Q$, then a
    maximal rectangle $R$ containing both $x$ and $y$ in $Q$ also contains $x+S$
    and $y+S$. Let $R'$ be the rectangle obtained from $R$ by moving each corner
    of $R$ towards its center along both the horizontal and the vertical axes by
    $\varepsilon$. Then $R'$ is contained in the region bounded by $P$, and both
    $x,y\in R'$. It follows that the center of every pixel of $P$ is $r$-visible to
    $Y$ (in $P$), thus $Y$ covers every pixel of $P$ due to
    \Cref{lemma:rvisionedge}.
\end{proof}
\begin{assumption}\label{def:rguard_restriction}
    We assume that $r$-guards in $P$ avoid the boundary $\partial(h\cap v)$ for
    every $hv\in E(G)$. Equivalently, an $r$-guard of $P$ must be contained in
    $\interior(h\cap v)$ for some $hv\in E(G)$.
\end{assumption}
Now \Cref{lemma:rvisionedge} implies that the region visible to an $r$-guard $x$
is determined by the unique edge $hv\in E(G)$ such that $x\in \interior(h\cap
v)$. In this way, an $r$-guard $x$ corresponds to the edge $hv\in E(G)$ whose
pixel contains it ($x\in h\cap v$). By \Cref{lemma:rvisionedge}, the
region $r$-visible to $x$ is the union of the pixels of edges of $G$ that
$r$-visible to $hv$.
Since a sliding camera covers the union of the visibility region of the
$r$-guards on its patrol, we may also make the following assumption without loss
of generality.
\begin{assumption}\label{def:sliding_restriction}
    The patrol of a horizontal sliding camera is not a subset of $\partial h$
    for any
    $h\in G_H$. The patrol of vertical sliding cameras is not a subset of $\partial v$
    for any $v\in G_V$.
    We assume that sliding cameras are maximal (in length).
\end{assumption}
\Cref{lemma:rvisionedge} implies that the
region covered by a horizontal or vertical sliding camera is determined by the unique
slice containing its patrol. For future reference, let us record the raster
graph adaptations of the $r$-guarding problem and sliding camera problems.
\begin{definition}\label{def:geometrytograph}
	Given a raster graph $G$, we call
	\begin{itemize}
		\item a set $F\subseteq E(G)$ a covering set of $r$-guards,
			if any edge $e\in E(G)$ is $r$-visible to some $f\in
			F$;
		\item a set $M_H\subseteq G_H$ a covering set of horizontal
			sliding cameras, if for any vertical slice $v\in G_V$
			there exists a horizontal slice $h\in M_H$ such that
			$hv\in E(G)$;
		\item a set $M_V\subseteq G_V$ a covering set of vertical
			sliding cameras, if for any horizontal slice $h\in G_H$
			there exists a vertical slice $v\in M_V$ such that
			$hv\in E(G)$;
        \item a set $S\subseteq G_H\cup G_V$ a covering set of sliding cameras,
            if for any edge $hv\in E(G)$ there exists a slice $s\in S$ such that
            $hs\in E(G)$ or $sv\in E(G)$. In other words, for any $s\in G_H\cup
            G_V$ we have $s\in \Gamma(S)$ or $\Gamma(s)\subseteq \Gamma(S)$.\footnote{A
            covering set of mobile guards was defined in~\cite{gyori_mobile_2019} as the
            notion corresponding to vertex dominating set in $G$.  That definition is unfortunately
    	    flawed because it does not correspond to the geometric definition.
            Note, that in grid intersection graphs, the domination number is
            indeed the appropriate notion for co-operative mobile
            guards~\cite{kosowski_cooperative_2007}.}
	\end{itemize}
\end{definition}

\begin{lemma}\label{lemma:equivalence}
    The notions defined by \Cref{def:geometrytograph} correspond to the
    geometric definitions under $r$-vision (with the restriction of
    \Cref{def:rguard_restriction,def:sliding_restriction}).
\end{lemma}
\begin{proof}
    The statement for $r$-guards follows from \Cref{lemma:rvisionedge}. The
    region visible to a maximal sliding camera is the union of the $r$-guards on
    its patrol. According to our previous statement, the region visible to a maximal horizontal sliding
    camera whose patrol is contained in $h\in G_H$ is equal to
    \begin{equation*}
        \bigcup_{v\in \Gamma(h)}\bigcup_{\substack{hv\boxdot h'v' \\ h'v'\in
        E(G)}} h'\cap v'=\bigcup_{v\in \Gamma(h)}\bigcup_{\substack{h'\in
                \Gamma(v)\\ v'\in\Gamma(h) \\ h'v'\in
        E(G)}} h'\cap v'=\bigcup_{\substack{h'\in
        \Gamma(\Gamma(h))\\ v'\in\Gamma(h) \\ h'\in \Gamma(v')}} h'\cap v'=
        \bigcup_{\substack{v'\in\Gamma(h) \\ h'\in \Gamma(v')}} h'\cap v'=
        \bigcup_{\substack{v'\in\Gamma(h)}} v'.
    \end{equation*}
    This proves the statement for horizontal sliding cameras, and by symmetry
    for vertical sliding cameras. A covering set of sliding cameras $S$ covers the
    pixel corresponding to $h'v'\in E(G)$ if and only if there exists $h\in S$
    such that $v'\in \Gamma(h)$ or there exists $v\in S$ such that $h'\in
    \Gamma(v)$.
\end{proof}

The concept of $R$-trees was introduced by~\citet{gyori_generalized_1996}.
\begin{definition}\label{def:Rtree}
The horizontal $R$-tree $T_H$ of $P$ is equal to
\begin{equation*}
    T_H:=\left(G_H,\Big\{\{h_1,h_2\}\subseteq G_H\ :\ h_1\neq h_2,\ h_1\cap
    h_2\neq\emptyset\Big\}\right),
\end{equation*}
i.e., $T_H$ is the intersection graph of the horizontal slices of $P$.
Similarly, the vertical $R$-tree of $P$ is the intersection graph of the vertical slices of $P$:
\begin{equation*}
    T_V:=\left(G_V,\Big\{\{v_1,v_2\}\subseteq G_V\ \mid\ v_1\neq v_2,\ v_1\cap
    v_2\neq\emptyset\Big\}\right).
\end{equation*}
\end{definition}

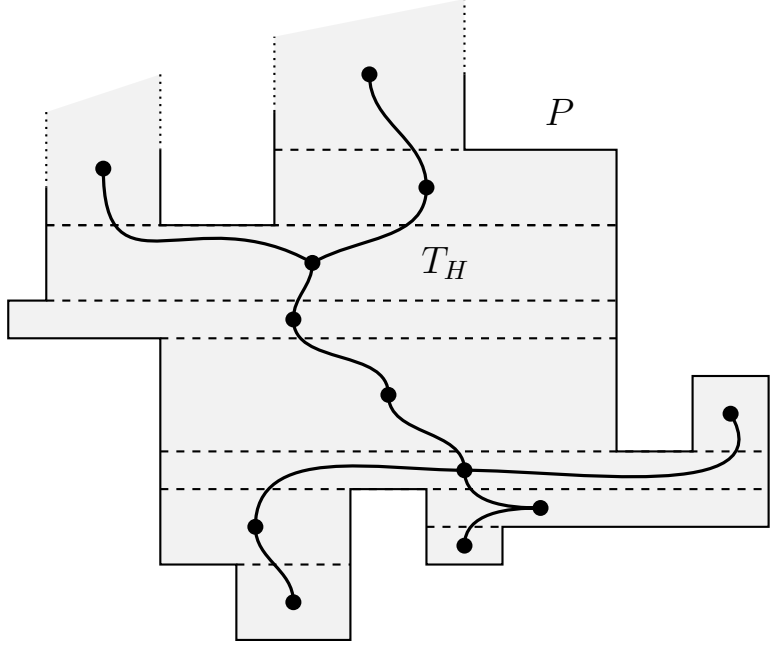
\begin{figure}[ht]
	\centering
	\begin{tikzpicture}
        \fill[gray,opacity=0.1] (-1.5,5)--(-1.5,2.5)--(-2,2.5)--(-2,2)--(0,2)--(0,-1)--(1,-1)--(1,-2)--(2.5,-2)--(2.5,0)--(3.5,0)--(3.5,-1)--(4.5,-1)
            --(4.5,-0.5)--(8,-0.5)--(8,1.5)--(7,1.5)--(7,0.5)--(6,0.5)--(6,4.5)--(4,4.5)--(4,6.5)
            -- (1.5,6)--(1.5,3.5)--(0,3.5)--(0,5.5);

        \draw[dotted] (-1.5,5)--(-1.5,4);
        \draw[dotted] (0,5.5)--(0,4.5);
        \draw[dotted] (1.5,5)--(1.5,6);
        \draw (1.5,5)--(1.5,3.5)--(0,3.5)--(0,4.5);
        \draw (-1.5,4)--(-1.5,2.5)--(-2,2.5)--(-2,2)--(0,2)--(0,-1)--(1,-1)--(1,-2)--(2.5,-2)--(2.5,0)--(3.5,0)--(3.5,-1)--(4.5,-1)
            --(4.5,-0.5)--(8,-0.5)--(8,1.5)--(7,1.5)--(7,0.5)--(6,0.5)--(6,4.5)--(4,4.5)--(4,5.5)
            edge[dotted] (4,6.5);

        \draw[dashed] (1.5,4.5)--(4,4.5);
        \draw[dashed] (-1.5,3.5)--(6,3.5);

        \draw[dashed] (-1.5,3.5)--(6,3.5);
        \draw[dashed] (-1.5,2.5)--(6,2.5);
        \draw[dashed] (0,2)--(6,2);
        \draw[dashed] (0,0.5)--(8,0.5);
        \draw[dashed] (0,0)--(8,0);
        \draw[dashed] (1,-1)--(2.5,-1);

        \draw[dashed] (3.5,-0.5)--(4.5,-0.5);

        \node[anchor=center] at (5.25,5) {\Large $P$};
        \node[anchor=center] at (3.75,3) {\Large $T_H$};

        \coordinate (a1) at (2.75,5.5);
        \coordinate (a2) at (-0.75,4.25);

        \coordinate (a3) at (3.5,4);
        \coordinate (a4) at (2,3);
        \coordinate (a5) at (1.75,2.25);
        \coordinate (a6) at (3,1.25);
        \coordinate (a7) at (4,0.25);
        \coordinate (a8) at (1.25,-0.5);
        \coordinate (a9) at (1.75,-1.5);
        \coordinate (a10) at (5,-0.25);
        \coordinate (a11) at (4,-0.75);
        \coordinate (a12) at (7.5,1);

        \foreach \x in {1,...,12}{\fill (a\x) circle[radius=3pt];}

            \draw[very thick,looseness=1.5] (a2) to[out=270,in=150] (a4);

            \draw[very thick,looseness=1] (a1) to[out=270,in=90] (a3)
                                            to[out=270,in=30] (a4)
                                            to[out=270,in=90] (a5)
                                            to[out=270,in=90] (a6)
                                            to[out=270,in=90] (a7)
                                            to[out=0,in=300,looseness=1] (a12);

            \draw[very thick,looseness=1] (a11) to[out=90,in=180] (a10)
                                                to[out=180,in=270] (a7)
                                                to[out=180,in=90] (a8)
                                                to[out=270,in=90] (a9);

	\end{tikzpicture}
    \caption{The simple orthogonal polygon $P$ is drawn with solid segments, and
        the region bounded by it is filled with gray. Each solid circle
        represents the horizontal slice enveloping it. The edges of $T_H$
    are drawn as curvy lines joining the solid circles representing its
    vertices. Edges of $T_H$ incident on slices not shown are not drawn either.}\label{fig:TH}
\end{figure}

\begin{figure}[ht]
	\centering
	\begin{tikzpicture}
        \fill[gray,opacity=0.1] (-1.5,5)--(-1.5,2.5)--(-2,2.5)--(-2,2)--(0,2)--(0,-1)--(1,-1)--(1,-2)--(2.5,-2)--(2.5,0)--(3.5,0)--(3.5,-1)--(4.5,-1)
            --(4.5,-0.5)--(8,-0.5)--(8,1.5)--(7,1.5)--(7,0.5)--(6,0.5)--(6,4.5)--(4,4.5)--(4,6.5)
            -- (1.5,6)--(1.5,3.5)--(0,3.5)--(0,5.5);

        \draw[dotted] (-1.5,5)--(-1.5,4);
        \draw[dotted] (0,5.5)--(0,4.5);
        \draw[dotted] (1.5,5)--(1.5,6);
        \draw (1.5,5)--(1.5,3.5)--(0,3.5)--(0,4.5);
        \draw (-1.5,4)--(-1.5,2.5)--(-2,2.5)--(-2,2)--(0,2)--(0,-1)--(1,-1)--(1,-2)--(2.5,-2)--(2.5,0)--(3.5,0)--(3.5,-1)--(4.5,-1)
            --(4.5,-0.5)--(8,-0.5)--(8,1.5)--(7,1.5)--(7,0.5)--(6,0.5)--(6,4.5)--(4,4.5)--(4,5.5)
            edge[dotted] (4,6.5);

        \draw[dashed] (-1.5,2.0)--(-1.5,2.5);
        \draw[dashed] (0,2)--(0,3.5);
        \draw[dashed] (1,-1)--(1,3.5);
        \draw[dashed] (1.5,-2)--(1.5,3.5);
        \draw[dashed] (2.5,0)--(2.5,6.2);
        \draw[dashed] (3.5,0)--(3.5,6.4);
        \draw[dashed] (4,-1)--(4,4.5);
        \draw[dashed] (4.5,-1)--(4.5,4.5);
        \draw[dashed] (6,-0.5)--(6,0.5);
        \draw[dashed] (7,-0.5)--(7,0.5);

        \node[anchor=center] at (5.25,5) {\Large $P$};
        \node[anchor=center] at (5.25,2) {\Large $T_V$};

        \coordinate (b1) at (-1.75,2.25);
        \coordinate (b2) at (-0.75,3.0);
        \coordinate (b3) at (0.5,1.25);
        \coordinate (b4) at (1.25,0.75);
        \coordinate (b5) at (2,1.75);
        \coordinate (b6) at (3,2.75);
        \coordinate (b7) at (3.75,2.25);
        \coordinate (b8) at (4.25,1.75);
        \coordinate (b9) at (5.25,1.25);
        \coordinate (b10) at (6.5,0);
        \coordinate (b11) at (7.5,1);

        \foreach \x in {1,...,11}{\fill (b\x) circle[radius=3pt];}

        \draw[very thick,looseness=1] (b1) to[out=0,in=180] (b2)
                                           to[out=0,in=90] (b3)
                                           to[out=270,in=180] (b4)
                                           to[out=0,in=270] (b5)
                                           to[out=90,in=180] (b6)
                                           to[out=0,in=90] (b7)
                                           to[out=270,in=180] (b8)
                                           to[out=0,in=90] (b9)
                                           to[out=270,in=180] (b10)
                                           to[out=0,in=270] (b11)
            ;
        
	\end{tikzpicture}
    \caption{The simple orthogonal polygon $P$ is drawn with solid segments, and
        the region bounded by it is filled with gray. Each solid circle
        represents the vertical slice enveloping it, and the edges of $T_V$
    are drawn as curvy lines joining the solid circles representing its
    vertices. Edges of $T_V$ incident on slices not shown are not drawn either.}\label{fig:TV}
\end{figure}

The graphs $T_H$ and $T_V$ are easily seen to be connected and cycle-free, because each
horizontal and vertical cut separates internally disjoint regions, as $P$ is
simple~\cite{gyori_short_1986,gyori_generalized_1996}. Equivalently, 
as shown on \Cref{fig:TH,fig:TV}, 
$T_H$ and $T_V$ are the duals of the planar graph determined by the segments of
$P$ and its horizontal and vertical cuts, respectively (without the outer face).

\begin{lemma}\label{lemma:pathG}
    For any $s\in G_H\cup G_V$, the neighborhood $\Gamma(s)$ is the vertex set of a path
    in the appropriate $R$-tree ($T_H$ or $T_V$).
\end{lemma}
\begin{proof}
    The vertical slices in $\Gamma(h)$ are internally disjoint
    and cover $h$. Because $h$ is a connected region, $\Gamma(h)$ induces a
    connected subtree in $T_V$. Furthermore, because $h$ is also vertically
    convex, $T_V[\Gamma(h)]$ cannot contain a vertex of degree 3 or greater.
    By symmetry, the statement holds for any $v\in G_V$, too.
\end{proof}

The statements of the following two lemmas are almost trivial, but stated
formally, they are not immediately obvious.
\begin{lemma}\label{lemma:boundary}
    Let $F\subseteq E(G)$ be a subset of edges in a raster graph $G$. Then
    \begin{equation*}
        \partial\left(\bigcup_{hv\in F}h\cap
        v\right)=\closure\left( \SymmDiff_{hv\in F}\partial\left(h\cap v\right)\right),
    \end{equation*}
    where $\Delta$ is the symmetric difference operator. The set of points
    exclusively contained by the closure 
    $\closure\left( \SymmDiff_{hv\in F}\partial\left(h\cap
    v\right)\right)\setminus\left( \SymmDiff_{hv\in F}\partial\left(h\cap v\right)\right)$ is a (finite)
    subset of the corners of pixels of $P$.
\end{lemma}
\begin{proof}
    Let $hv\in F$ and consider one of the four sides of the
    rectangle $h\cap v$: there exists at most one other $h'v'\in F$ such that $h'v'$ also
    contains this side. Thus the symmetric difference operator on the right hand
    side produces the
    interior of those sides that appear precisely in one $\partial(h\cap v)$ for $hv\in F$.
    There are four corners of $\partial(h\cap v)$ for each $hv\in F$. The symmetric
    difference operator produces such a corner $A$ if
    and only there is exactly one $hv\in F$ or there are exactly three $hv\in F$
    such that $A\in \partial(h\cap v)$. The set $\partial(\bigcup_{hv\in F}h\cap
        v)$ contains a finite
    number of corners of $\partial(h\cap v)$ for all $hv\in F$, therefore the
    closure re-introduces only a finite number of missing points.
\end{proof}

\begin{lemma}\label{lemma:convexHull}
    Let $v\in G_V$ and $h_1,h_2\in \Gamma(v)$. Let $[h_1,h_2]$ denote the
    vertex set of the path that connects $h_1$ to $h_2$ in $T_H$.
	Then
	\begin{equation*}
		\bigcup_{h\in [h_1,h_2]} h\cap v=\mathrm{convexHull}\Big( (h_1\cap v)\cup
		(h_2\cup v)\Big).
	\end{equation*}
\end{lemma}
\begin{proof}
    By \Cref{lemma:pathG}, $\Gamma(v)$ is the vertex set of a path in $T_H$,
    thus $[h_1,h_2]\subseteq \Gamma(v)$. The statement follows from the
    convexity of $v$.
\end{proof}

In this section, we defined the raster graph $G$ as the intersection graph of the vertex set of the
horizontal and vertical $R$-trees of a polygon.
Let us present an abstract generalization of raster graphs.
\begin{definition}\label{def:basegraph}
    We call $G$ a \emph{tree-based bipartite graph} with $T_H$ and $T_V$ as its horizontal and vertical
    $R$-trees if the following statements hold:
    \begin{enumerate}[label={(\roman*)}]
        \item $G=(G_H,G_V,E)$ is a connected bipartite graph;
        \item $T_H$ is a tree whose vertex set is $G_H$;
        \item $T_V$ is a tree whose vertex set is $G_V$;
        \item for any $h\in G_H$, the
            neighborhood $\Gamma(h)$ is the vertex set of a path in $T_V$;
        \item for any $v\in G_V$, the
            neighborhood $\Gamma(v)$ is the vertex set of a path in $T_H$.
    \end{enumerate}
    Without a subscript, $\Gamma(s)$ still denotes the set of neighbors of $s$ in $G$.
\end{definition}
Every raster graph is a tree-based bipartite graph. In the following
Sections~\ref{sec:partial}~to~\ref{sec:efficient} we will work with an arbitrary
tree-based bipartite graph $G$. We will return to working on raster graphs and
proving \Cref{thm:alg_typeA} in \Cref{sec:raster_quadratic}.

\section{The partial order on the vertices}\label{sec:partial}
Let $G$ be a tree-based bipartite graph with $R$-trees $T_H$ and $T_V$ (see
\Cref{def:basegraph}). In this section we pour the foundation of this paper,
that is, we develop the elementary lemmas we will use in later sections to
analyze the structure of $G$.

\begin{definition}[Roots of the $R$-trees]\label{def:roots}
    Let $h_\mathrm{root}v_\mathrm{root}$ be an edge of $G$ such that $v_\mathrm{root}$ is
    a leaf in $T_V$. The vertices $h_\mathrm{root}$ and $v_\mathrm{root}$ will
    serve as the roots $T_H$ and $T_V$, respectively.
\end{definition}
\begin{remark}\label{remark:roots} Although the choice of the roots potentially
    breaks the symmetry between $T_H$ and $T_V$ (there may not exist an edge
    whose endpoints are both leaves in the $R$-trees), all of our results and
    theorems, except for \Cref{lemma:leaf_comparable}, are agnostic to this
    symmetry breaking. The two applications of \Cref{lemma:leaf_comparable} are
    in the proofs \Cref{lemma:inheriting_neighbor,lemma:ell_gamma_convex}.
    However, the statements of these lemmas are completely symmetric with
    respect to $T_H$ and $T_V$. In other words, every lemma and theorem in this
    paper holds even if we swap the roles of $G_H$ and $G_V$.
\end{remark}

\begin{definition}\label{def:parent}
    For any $h\in G_H\setminus \{h_\mathrm{root}\}$, let $\mathrm{parent}(h)$ be
    the \emph{parent} of $h$ in $T_H$. Similarly, for any $v\in G_H$, let
    $\mathrm{parent}(v)$ be the \emph{parent} of $v$ in $T_V$ if $v\neq
    v_\mathrm{root}$. Define
    $\mathrm{parent}(h_\mathrm{root})=\emptyset$ and $\mathrm{parent}(v_\mathrm{root})=\emptyset$.
\end{definition}

\begin{definition}\label{def:partial}
    Given $s_1,s_2\in G_H$ or $s_1,s_2\in G_V$, let
    $[s_1,s_2]$ be the vertex set of the path joining them in $T_H$ or
    $T_V$, respectively. Define $\partial [s_1,s_2]:=\{s_1,s_2\}$,
    $\partial(\emptyset):=\emptyset$.
\end{definition}

\begin{definition}\label{def:order}
    Let us define a partial order on $G_H$ and $G_V$. For any $h_1,h_2\in G_H$
    and any $v_1,v_2\in G_V$, define:
    \begin{align*}
        h_1\le h_2 &\Longleftrightarrow h_1\in [h_\mathrm{root},h_2]\\
        v_1\le v_2 &\Longleftrightarrow v_1\in [v_\mathrm{root},v_2]
    \end{align*}
    The vertices $s_1,s_2\in G_H\cup G_V$ are \emph{comparable} if and only if
    $s_1\le s_2$ or $s_2\le s_1$. An element of $G_H$ is not comparable to any
    element of $G_V$, and vice versa.
\end{definition}
Using $<$ and $\le$ denotes both the partial order on the vertices of $G$ and the natural
order on the integers in this paper. However, the one which we are referring to
should always be clear from context; in addition, vertices are usually denoted by
(an indexed) $h,v,t,s$, while the variables $i,j,k,m,n$ stand for integers.

Later we will introduce an extension of this partial order
(\Cref{def:unifiedorder}). However, unless otherwise stated, when discussing the
comparability of slices or referring to an order on the slices, it is in the sense
of \Cref{def:order}.

\begin{definition}\label{def:minmax}
    Given a set of elements $S\subseteq G_H$ or $S\subseteq G_V$, if the unique minimum
    and unique maximum elements of the set exist (with respect to the order
    defined above), then we denote them by $\min S$
    and $\max S$, respectively. If there is no unique minimal (maximal)
    element, then we say that the minimum (maximum) does not exist. For
    convenience, we define $\min\emptyset=\max\emptyset=\emptyset$.
\end{definition}

\begin{lemma}\label{lemma:connected_convex}
    Suppose $S$ induces a connected subgraph in $T_H$ or $T_V$. If $a,b\in S$,
    then $[a,b]\subseteq S$.
\end{lemma}

\begin{corollary}\label{lemma:gamma_convex}
    Let $s_0\in G_V\cup G_H$. If $s_1,s_2\in \Gamma(s_0)$, then $[s_1,s_2]\in \Gamma(s_0)$.
\end{corollary}

\begin{corollary}\label{lemma:intersection}
    Let $s_1,s_2\in G_V\cup G_H$ such that
    $\Gamma(s_1)\cap\Gamma(s_2)\neq\emptyset$. Then
    $\Gamma(s_1)\cap\Gamma(s_2)$ induces a path in $T_H$ or $T_V$.
\end{corollary}

\begin{corollary}\label{lemma:connected_min}
    If $S$ induces a connected subgraph in $T_H$ or $T_V$, then $\min S$ exists.
\end{corollary}

\begin{lemma}\label{lemma:le_comparable}
    Suppose $a,b,c\in G_H$ or $a,b,c\in G_V$.
    If $a,b\le c$, then $a$ and $b$ are comparable.
\end{lemma}
\begin{proof}
    Since $a,b\le c$, both vertices $a,b$ lie on the path from the root to $y$.
\end{proof}

\begin{lemma}\label{lemma:Tneighbors_gamma_intersect}
    If $s_1,s_2$ are neighbors in $T_H$ or $T_V$, then
    $\Gamma(s_1)\cap\Gamma(s_2)\neq\emptyset$.
\end{lemma}
\begin{proof}
    Since $G$ is connected, there is a path in $G$ on vertices
    $t_1t_2\ldots t_{2k+1}$ such that $t_1=s_1$ and $t_2=s_2$. Notice that
    $\cup_{i=1}^k\Gamma(t_{2i})$ induces a connected subgraph of $T_H$ or $T_V$,
    therefore there exists some $i$ such that $s_1,s_2\in\Gamma(t_i)$.
\end{proof}

\begin{corollary}\label{lemma:gamma_preserves_connected}
    If $S$ induces a connected subgraph in $T_H$ or $T_V$, then so does
    $\Gamma(S)$.
\end{corollary}

\begin{definition}
    For the sake of legibility, we will denote
    \begin{equation*}
    \Gamma\Gamma(s):=\Gamma(\Gamma(s)).
    \end{equation*}
\end{definition}

\begin{corollary}\label{cor:minGamma_exists}
    The unique minimums $\min \Gamma(s)$, $\min \Gamma\Gamma(s)$, $\min
    \Gamma\Gamma\Gamma(s)$ exist for any $s\in G_H\cup G_V$.
\end{corollary}

\begin{corollary}\label{cor:minGamma_le_neighbor}
    For any $s\in G_H\cup G_V$ and $t\in \Gamma(s)$, we have $\min\Gamma(s)\le
    t$.
\end{corollary}

\begin{lemma}\label{lemma:leaf_comparable}
    If $s_1,s_2\in \Gamma(s_0)$ and $\min\Gamma(s_0)=v_\mathrm{root}$, then
    $s_1$ and $s_2$ are comparable.
\end{lemma}
\begin{proof}
    Since $v_\mathrm{root}$ is a leaf, we must have $\min\Gamma(s_0)\in
    \partial\Gamma(s_0)$. Thus $\Gamma(s_0)=[v_\mathrm{root},v]$, where
    $v\in\partial\Gamma(s_0)$. It follows that both $s_1,s_2\le v$, so they must be
    comparable.
\end{proof}

\begin{lemma}\label{lemma:minsarecomparable}
    If $\Gamma(s_1)\cap
    \Gamma(s_2)\neq \emptyset$, then $\min\Gamma(s_1)$ and $\min\Gamma(s_2)$ are
    comparable.
\end{lemma}
\begin{proof}
    Follows from \Cref{lemma:le_comparable}.
\end{proof}

\begin{lemma}\label{lemma:minless_comparable}
    If $\min\Gamma(s_1)< \min\Gamma(s_2)$, then any two elements in
    $\Gamma(s_1)\cap \Gamma(s_2)$ are comparable.
\end{lemma}
\begin{proof}
    Let $s,s'\in \Gamma(s_1)\cap \Gamma(s_2)$ be arbitrary. By definition
    \begin{equation*}
        \min\Gamma(s_1)< \min\Gamma(s_2)\le s,s'.
    \end{equation*}
    Notice, however, that there is a unique $s\in\partial\Gamma(s_1)$ such that
    $\min\Gamma(s_2)\le s$, therefore any two elements in $\Gamma(s_1)\cap \Gamma(s_2)$
    are less than or equal to $s$.
\end{proof}

\begin{corollary}\label{lemma:incomparable_common_neighbors}
    If $s_1$ and $s_2$ are not comparable, then any $t_1,t_2\in \Gamma(s_1)\cap
	\Gamma(s_2)$ satisfies $\min\Gamma(t_1)=\min\Gamma(t_2)$.
\end{corollary}

\begin{lemma}\label{lemma:min_in_other_gamma}
    If $\Gamma(s_1)\cap \Gamma(s_2)\neq \emptyset$ and
    $\min\Gamma(s_1)\le\min\Gamma(s_2)$, then $\min\Gamma(s_2)\in
    \Gamma(s_1)$.
\end{lemma}
\begin{proof}
    Let $s_0\in
    \Gamma(s_1)\cap\Gamma(s_2)$, then
    \begin{equation*}
        \min\Gamma(s_1)\le \min\Gamma(s_2)\le s_0,
    \end{equation*}
    and therefore
    \begin{equation*}
        \min\Gamma(s_2)\in [\min\Gamma(s_1),s_0]\subseteq \Gamma(s_1),
    \end{equation*}
    which proves the claim.
\end{proof}

The following lemma is essential to all of our later proofs. It reinforces and
ties the foundations together.
\begin{lemma}\label{lemma:minle}
    If $s_1\le s_2$, then $\min\Gamma(s_1)\le\min\Gamma(s_2)$.
\end{lemma}
\begin{proof}
    We claim that it is sufficient to prove that if $v_1v_2\in E(T_V)$ and $v_1< v_2$, then
    $\min\Gamma(v_2)\in \Gamma(v_1)$. Indeed, if proved, the lemma follows in
    full generality by transitivity of the partial order (and by symmetry for
    $G_H$) and \Cref{cor:minGamma_exists}.

	\medskip

	Let us suppose that $v_1v_2\in E(T_V)$ and $v_1<v_2$. By
    \Cref{lemma:Tneighbors_gamma_intersect}, we have
    $\Gamma(v_1)\cap\Gamma(v_2)\neq\emptyset$. By
    \Cref{lemma:min_in_other_gamma}, we must have $\min\Gamma(v_1)\in
    \Gamma(v_2)$ or $\min\Gamma(v_2)\in \Gamma(v_1)$. If the latter holds, then
    we are done. If the latter does not hold, then $\min\Gamma(v_2)<\min\Gamma(v_1)$.

    \medskip

    Suppose that $v_1v_2\in E(T_H)$, $v_1<v_2$, $\min\Gamma(v_2)\not\in
    \Gamma(v_1)$ and $v_2$ is minimal with respect to these condition. If
    $v_1=v_\mathrm{root}$, then $h_\mathrm{root}=\min\Gamma(v_1)\in \Gamma(v_2)$,
    thus $\min\Gamma(v_2)=h_\mathrm{root}$, so $\min\Gamma(v_2)\in \Gamma(v_1)$,
    which is a contradiction. If $v_1\neq v_\mathrm{root}$, then let
    $v_0=\mathrm{parent}(v_1)$. By the minimality of $v_2$, we must have
    $\min\Gamma(v_1)\in \Gamma(v_0)$ and $\min\Gamma(v_0)\le \min\Gamma(v_1)$.
    By \Cref{lemma:le_comparable}, $\min\Gamma(v_0)$ and $\min\Gamma(v_2)$ are
    comparable.

    If $\min\Gamma(v_0)\le \min\Gamma(v_2)$, then $\min\Gamma(v_2)\in
    [\min\Gamma(v_0),\min\Gamma(v_1)]\subseteq \Gamma(v_0)$. But then
    $v_1\in [v_0,v_2]\subseteq \Gamma(\min\Gamma(v_2))$, which is a
    contradiction. Therefore we must have
    \begin{equation*}
        \min\Gamma(v_2)<\min\Gamma(v_0).
    \end{equation*}
    By induction (take the parent of $v_0$, if it exists), it follows that
    $\min\Gamma(v_2)<\min\Gamma(v_\mathrm{root})=h_\mathrm{root}$, which is
    a contradiction.
\end{proof}

\begin{corollary}\label{cor:minGG}
	For any $s\in G_H\cup G_V$ we have
	\begin{equation*}
		\min \Gamma\Gamma(s)=\min \Gamma(\min\Gamma(s)).
	\end{equation*}
\end{corollary}

\begin{corollary}\label{cor:minGG2}
	If $\min\Gamma(s_1)\le \min\Gamma(s_2)$, then $\min \Gamma\Gamma(s_1)\le \min\Gamma\Gamma(s_2)$.
\end{corollary}

\begin{lemma}\label{lemma:minGGcloser}
    For any $s\in G_H\cup G_V$, either $\min\Gamma\Gamma(s)<s$ or $s$ is the
    root of $T_H$ or $T_V$.
\end{lemma}
\begin{proof}
    If $s$ is not the root of the appropriate $R$-tree, let $s_0$ be its parent.
    Since $G$ is connected, there exists a vertex that is joined to both $s$ and
    $s_0$, thus $\min\Gamma\Gamma(s)\le s_0<s$.
\end{proof}

\begin{lemma}\label{lemma:helly2}
    If $s_1,s_2,s_3\in \Gamma(s_0)$ and $\min\Gamma(s_0)<s_1\le s_2,s_3$, then
    $s_2$ and $s_3$ are comparable.
\end{lemma}
\begin{proof}
    Let $\{a,b\}=\partial\Gamma(s_0)$. If $s_2,s_3\le a$ or  $s_2,s_3\le b$,
    then the statement follows from \Cref{lemma:le_comparable}. If $s_2\le a$
    and $s_3\le b$, then $s_1\le a,b$ and so $s_1\le \min\Gamma(s_0)$, which is
    a contradiction.
\end{proof}

\section{Refining the partial order}\label{sec:refined}
\Cref{alg:main}, described in \Cref{sec:cover}, is a greedy algorithm that
processes elements of $G_V$ one by one. However, we need to refine the $<$
order, such that the new partial order corresponds to the intuitive notion of sweeping.
\begin{definition}\label{def:unifiedorder}
    Let us define the partial order $<_b$ on $G_H\cup G_V$ as follows:
\begin{itemize}
    \item if $h_1,h_2\in G_H$, then let
    \begin{align*}
        h_1<_b h_2 \Longleftrightarrow\ &\min\Gamma(h_1)<\min\Gamma(h_2)\text{\ or }\\
        &\min\Gamma(h_1)=\min\Gamma(h_2)\text{\ and }h_1<h_2;
    \end{align*}
    \item if $v_1,v_2\in G_V$, then let
    \begin{align*}
        v_1<_b v_2 \Longleftrightarrow\ &\min\Gamma(v_1)<\min\Gamma(v_2)\text{\ or }\\
        &\min\Gamma(v_1)=\min\Gamma(v_2)\text{\ and }v_1<v_2;
    \end{align*}
    \item if $h\in G_H$ and $v\in G_V$, then $h$ and $v$ are not comparable.
\end{itemize}
Naturally, we define $s_1\le_b s_2\Longleftrightarrow (s_1=s_2\text{ or }s_1<_b s_2)$.
\end{definition}

\begin{lemma}\label{lemma:partial_order_b}
    The relation $<_b$ is a partial order that extends $<$.
\end{lemma}
\begin{proof}
    The extension property follows from \Cref{lemma:minle} and \Cref{cor:minGG2}.
    The irreflexivity of $<_b$ follows from the irreflexivity of $<$. The asymmetry of
    $<_b$ also follows from the asymmetry of $<$. Transitivity of $<_b$ follows
    from the transitivity of $<$ and \Cref{cor:minGG,cor:minGG2}.
\end{proof}

It is well-known that any partial order can be extended to a total or linear
order.
\begin{definition}[Linear order on the vertices]\label{def:totalorder}
    Let $\prec$ be a refinement of $<_b$ on $G_H\cup G_V$, such that $\prec$
    is a linear order when restricted to $G_H$ and $G_V$. If $h\in G_H$ and
    $v\in G_V$, then $h$ and $v$ are not comparable with respect to $\prec$.
\end{definition}

\begin{lemma}\label{lemma:linear_order_min_Gamma}
    Suppose $\min\Gamma(s)$ and $\min\Gamma(s')$ are comparable. If $s\prec
    s'$, then $\min\Gamma(s)\le \min\Gamma(s')$.
\end{lemma}
\begin{proof}
	If $\min\Gamma(s')<\min\Gamma(s)$, then by definition, $s'<_b s$, which
	contradicts $s\prec s'$.
\end{proof}

\begin{lemma}\label{lemma:linear_order_intersecting_neighborhoods}
    If $s\prec s'$ and $\Gamma(s)\cap \Gamma(s')\neq\emptyset$, then
    $\min\Gamma(s')\in \Gamma(s)$ and $\min\Gamma(s)\le\min\Gamma(s')$.
\end{lemma}
\begin{proof}
    Follows from \Cref{lemma:minsarecomparable,lemma:linear_order_min_Gamma}.
\end{proof}

The following two lemmas will play a very important role later on.
\Cref{lemma:inheriting_neighbor} shows that neighborhoods are co-descending with
respect to $\prec$. A possible interpretation of \Cref{lemma:ell_gamma_convex}
is that ``holes'' are not allowed. The proofs of the next two lemmas are mildly
technical, so they are postponed to \Cref{sec:postponed_refined}.
\begin{lemma}\label{lemma:inheriting_neighbor}
    Let $h_1,h_2\in G_H$ and $v_1,v_2\in G_V$ be such that $h_1\preceq h_2$ and
    $v_1\preceq v_2$. If $h_1,h_2\in \Gamma(v_2)$ and $h_2\in \Gamma(v_1)$,
    then $h_1\in \Gamma(v_1)$.
\end{lemma}

\begin{lemma}\label{lemma:ell_gamma_convex}
    Let $h_1,h_2,h_3\in G_H$ and $v_1,v_2,v_3\in G_V$ be such that $h_1\preceq
    h_2\preceq h_3$ and $v_1\preceq v_2\preceq v_3$. If $h_2,h_3\in
    \Gamma(v_1)$, and $h_1,h_3\in \Gamma(v_2)$, and $h_1,h_2\in \Gamma(v_3)$
    then $h_2\in \Gamma(v_2)$.
\end{lemma}
Note that these lemmas are completely symmetric with respect to exchanging $G_V$
and $G_H$ (see~\Cref{fig:two_lemmas}), even though the lemmas are formally
stated asymmetrically.
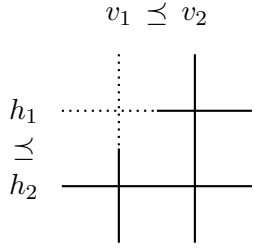
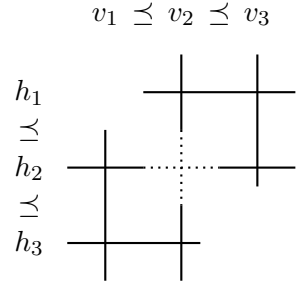
\begin{figure}
	\centering
	\begin{subfigure}[t]{0.4\textwidth}
	\centering
	\begin{tikzpicture}
		\draw[dotted] (1,2.75) -- (1,1.5);
		\draw (1,1.5) -- (1,0.25);
		\draw (2,2.75)--(2,0.25);

		\draw[dotted] (0.25,2)--(1.5,2);
		\draw (1.5,2)--(2.75,2);
		\draw (0.25,1)--(2.75,1);

		\node[above] at (1,3) {$v_1$};
		\node[above] at (1.5,3) {$\preceq$};
		\node[above] at (2,3) {$v_2$};

		\node at (-0.25,2) {$h_1$};
		\node at (-0.25,1.5) {$\preceq$};
		\node at (-0.25,1) {$h_2$};
	\end{tikzpicture}
	\caption{Given the assumptions and
		premises of \Cref{lemma:inheriting_neighbor},
	one can conclude that $h_1v_1\in E(G)$.}\label{fig:lemma:inheriting_neighbor}
	\end{subfigure}%
	\hfill%
	\begin{subfigure}[t]{0.4\textwidth}
	\centering
	\begin{tikzpicture}

		\draw (2,3.5)--(2,1.75);
		\draw (1,3.5) -- (1,2.5);
		\draw[dotted] (1,2.5) -- (1,1.5);
		\draw (1,1.5) -- (1,0.5);
		\draw (0,2.5)--(0,0.5);

		\draw (-0.5,1) -- (1.25,1);
		\draw (-0.5,2) -- (0.5,2);
		\draw[dotted] (0.5,2)--(1.5,2);
		\draw (1.5,2)--(2.5,2);
		\draw (0.5,3)--(2.5,3);

		\node[above] at (0,3.75) {$v_1$};
		\node[above] at (0.5,3.75) {$\preceq$};
		\node[above] at (1,3.75) {$v_2$};
		\node[above] at (1.5,3.75) {$\preceq$};
		\node[above] at (2,3.75) {$v_3$};

		\node at (-1,3) {$h_1$};
		\node at (-1,2.5) {$\preceq$};
		\node at (-1,2) {$h_2$};
		\node at (-1,1.5) {$\preceq$};
		\node at (-1,1) {$h_3$};
	\end{tikzpicture}
    \caption{Given the assumptions and premises of
    \Cref{lemma:ell_gamma_convex}, one can conclude that $h_2v_2\in
    E(G)$.}\label{fig:lemma:ell_gamma_convex}
	\end{subfigure}
	\caption{Visualizing
		\Cref{lemma:inheriting_neighbor,lemma:ell_gamma_convex}. The vertical
		lines represent $v_1,v_2$ (and $v_3$), while the horizontal lines
		represent $h_1,h_2$ (and $h_3$). An intersection of the solid sections of two
		lines indicates that the represented elements are assumed to be joined
		by an edge in the premise of the appropriate lemma. 
		The two lines that intersect with their dotted sections represent the
		vertices of the edge that appears in the conclusion of
		each lemma.}\label{fig:two_lemmas}
\end{figure}

\section{Guards, covers, and independence}\label{sec:guards}
Let us define $r$-guards and $r$-independence with respect to the tree-based
bipartite graph
$G$. \Cref{lemma:indep} explains the abstract definition of $r$-independence.
Recall \Cref{def:rvisionedge,def:geometrytograph,def:basegraph}. Let us explicitly define
$r$-guards in a tree-based bipartite graph $G$.
\begin{definition}\label{def:rguard}
    An $r$-guard in $G$ is an edge $hv\in E(G)$. The $r$-guard $hv$ covers an edge
    $h'v'\in E(G)$ if and only if $h'v,hv'\in E(G)$.
\end{definition}

\begin{definition}\label{def:independent}
    Two edges $h_1v_1$ and $h_2v_2$ are $r$-independent in $G$ if and only if
    $\Gamma(h_1)\cap \Gamma(h_2)=\emptyset$ or $\Gamma(v_1)\cap \Gamma(v_2)=\emptyset$.
    If two edges are not $r$-independent, then we call them $r$-dependent.
\end{definition}
\begin{lemma}\label{lemma:indep}
    Two edges are $r$-independent in $G$ if and only if there does not exist an
    $r$-guard $e\in E(G)$ that covers both edges.
\end{lemma}
\begin{proof}
    If $h_0v_0\in E(G)$ covers both $h_1v_1$ and $h_2v_2$, then $h_0\in
    \Gamma(v_1)\cap\Gamma(v_2)$ and $v_0\in \Gamma(h_1)\cap\Gamma(h_2)$. Therefore $h_1v_1$ and
    $h_2v_2$ are not $r$-independent if they are both covered by $h_0v_0$.

    Suppose that $h_1v_1$ and $h_2v_2$ are not $r$-independent. If $v_2\prec
    v_1$, then exchange $h_1v_1$ and $h_2v_2$. Thus without loss of generality,
    we may assume that $v_1\preceq v_2$. Since
    $\Gamma(v_1)\cap\Gamma(v_2)\neq\emptyset$,
    \Cref{lemma:linear_order_intersecting_neighborhoods} implies that
    $\min\Gamma(v_2)\in \Gamma(v_1)$.

    Because $v_2\in \Gamma(h_2)$ and
    $h_2\in \Gamma(v_2)$, we have $\min\Gamma(v_2)\le
    h_2$ and $\min\Gamma(h_2)\le v_2$. It follows from
    \Cref{lemma:inheriting_neighbor} that
    $e=\min\Gamma(v_2)\min\Gamma(h_2)$ is an edge in $E(G)$.

    \begin{itemize}
        \item If $h_1\preceq h_2$:
            \Cref{lemma:linear_order_intersecting_neighborhoods} implies that
            $\min\Gamma(h_2)\in \Gamma(h_1)$.  By the previous reasoning, $e$
            covers $h_1v_1$, and $e$ trivially covers $h_2v_2$.

        \item If $h_2\prec h_1$:
            \Cref{lemma:linear_order_intersecting_neighborhoods} implies that
            $\min\Gamma(h_1)\in \Gamma(h_2)$.

            Note, that $\min\Gamma(v_2)\le h_2\prec h_1$
            and $\min\Gamma(h_1)\le v_1\prec v_2$. We have $h_2,h_1\in
            \Gamma(\min\Gamma(h_1))$, $\min\Gamma(v_2),h_1\in\Gamma(v_1)$, and
            $\min\Gamma(v_2),h_2\in \Gamma(v_2)$. Now
            \Cref{lemma:ell_gamma_convex} implies $h_2\in \Gamma(v_1)$. It
            follows that $h_2v_1\in E(G)$ covers both $h_1v_1$ and $h_2v_2$.
    \end{itemize}
\end{proof}
We will construct a covering set of $r$-guards of $G$ gradually in
\Cref{alg:main}, adding or replacing at most one guard in each
iteration. Because we want a kind of monotonicity to hold for the set of covered
edges, we need to put artificial restrictions on what is considered covered by a guard.
\begin{definition}\label{def:semiguard}
    A \emph{semi-guard} is an ordered pair of vertices $(h,v)\in G_H\times G_V$ such that $\min
    \Gamma(v)$ and $\min \Gamma(h)$ are joined by an edge in $G$. The vertices $h$
    and $v$ are called the elements of the semi-guard. As a
    guard, it covers a subset of the edges covered by the edge $\min \Gamma(v)\min
    \Gamma(h)\in E(G)$:
    \begin{equation*}
        (h,v)\text{\ covers }h'v'\in E(G)\quad\Longleftrightarrow\quad
        \left\{\begin{array}{l}
         h'\preceq h,\\  v'\preceq v,\\
        \min\Gamma(h)\in \Gamma(h'),\\
        \min\Gamma(v)\in \Gamma(v').
        \end{array}
        \right.
    \end{equation*}
\end{definition}
\begin{remark}
    Technically, we will not need the requirement $v'\preceq v$ in the proofs,
    because the elements of $G_V$ will be processed in decreasing $\prec$ order
    anyway. However, the symmetric definition explains some of the choices in
    the design of \Cref{alg:main}, therefore we keep it for conceptual and
    didactic reasons. For example, the following lemma provides a nice
    equivalent formulation for the set of edges covered by a semi-guard.
\end{remark}
\begin{lemma}\label{lemma:alter_def_semiguard}
    Let $(h,v)$ be a semi-guard and let $h'v'\in E(G)$. Then
    \begin{equation*}
        (h,v)\text{\normalfont\ covers }h'v'\in E(G)\quad\Longleftrightarrow\quad\left\{
        \begin{array}{l}
            h'\preceq h,\\
            v'\preceq v,\\
            \Gamma(h')\cap\Gamma(h)\neq\emptyset,\\
            \Gamma(v')\cap\Gamma(v)\neq\emptyset.
        \end{array}
        \right.
    \end{equation*}
\end{lemma}
\begin{proof}
    Follows from \Cref{lemma:linear_order_intersecting_neighborhoods}.
\end{proof}

\begin{definition}\label{def:covers}
    A set of semi-guards $W$ is said to \emph{cover} an edge $e\in E(G)$ if $\exists
    (h,v)\in W$ that covers $e$. A set of semi-guards $W$ is said to
    \emph{cover} a subset of edges $F\in E(G)$ if every edge $f\in F$ is covered
    by $W$.
\end{definition}
To prove that a set of $r$-guards (or semi-guards) that covers every
edge of $G$ has the minimum possible size, it is sufficient to construct
a set of $r$-independent edges of the same cardinality. Indeed, this will be
our strategy with \Cref{alg:main,alg:independent}.

\begin{lemma}\label{lemma:semiguard_necessary}
    If $(h,v)$ is a semi-guard, then
    \begin{align*}
        \min\Gamma\Gamma(h)&\le \min\Gamma(v),\\
        \min\Gamma\Gamma(v)&\le \min\Gamma(h).
    \end{align*}
\end{lemma}
\begin{proof}
    If $(h,v)$ is a semi-guard, then by definition, $\min\Gamma(v)\in
    \Gamma(\min\Gamma(h))$ and $\min\Gamma(h)\in \Gamma(\min\Gamma(v))$. The
    inequalities now follow from \Cref{cor:minGG}.
\end{proof}

The following lemma shows that beyond the trivial necessary conditions,
verifying a simple condition is sufficient to conclude that a semi-guard covers an edge.
\begin{lemma}\label{lemma:semiguard_covers}
    Let $(h,v)$ be a semi-guard, and let $h'v'\in E(G)$ be such that $v'\preceq
    v$ and $h'\preceq h$. If $h\in \Gamma(v')$ holds, then $(h,v)$ covers $h'v'$.
\end{lemma}
\begin{proof}
    If $h\in \Gamma(v')$, then $v'$ is a common neighbor of $h$ and $h'$. By
    \Cref{lemma:linear_order_intersecting_neighborhoods}, $\min\Gamma(h)\in
    \Gamma(h')$ also holds. Similarly, $h$ is a common neighbor of $v$ and
    $v'$, therefore $\min\Gamma(v)\in \Gamma(v')$.
\end{proof}

\begin{lemma}\label{lemma:semiguard_hv_edge}
	Suppose $hv\in E(G)$. Then $(h,v)$ is a semi-guard, and $(h,v)$ covers
    \begin{equation*}
        \{ h'v \in E(G)\mid h'\preceq h\}.
    \end{equation*}
\end{lemma}
\begin{proof}
	Note that $h\in \Gamma(v)$, so we have
	$\min\Gamma(v)\le h$ and $v\in \Gamma(h)\cap \Gamma(\min\Gamma(v))$.
    By \Cref{lemma:linear_order_intersecting_neighborhoods}, we have
    $\min\Gamma(h)\in \Gamma(\min\Gamma(v))$, so $(h,v)$ is a semi-guard. By taking $v=v'$,
    the rest of the lemma follows from \Cref{lemma:semiguard_covers}.
\end{proof}

\section{Constructing a minimum cover}\label{sec:cover}
We are ready to describe \Cref{alg:main}, which constructs a set of semi-guards
that cover a tree-based bipartite graph $G$.
\begin{definition}
	Let ${(s_i)}_{i=1}^k$ be the sequence of the elements of $G_V$ in decreasing $\prec$
	order, that is,
	\begin{equation*}
		s_k\prec s_{k-1}\prec \cdots \prec s_2\prec s_1.
	\end{equation*}
\end{definition}
Informally, \Cref{alg:main} proceeds in the following manner:
\begin{enumerate}
	\item \Cref{alg:main} ``sweeps'' through the vertices of $G_V$ in $k$ rounds
        in decreasing $\prec$ order (\Cref{def:totalorder}). That is, the main
        loop visits the elements $s_i$ one by one, as $i$ enumerates the
        integers from $1$ to $k$.

    \item In every round, $t_i\in \Gamma(s_i)$ is the maximum $\prec$ order such
        that $t_is_i$ is not covered by $W_{i-1}$ (if such a $t_i$ does not
        exist, then we define $t_i=\emptyset$ and move on to $s_{i+1}$).
        If the current vertex $s_i\in G_V$ is incident with a yet
		uncovered edge $t_{i}s_{i}$ , then \Cref{alg:main} tries to modify an
		existing semi-guard $(h,v)\in W_{i-1}$ to cover $t_{i}s_{i}$, see
		\cref{line:modify}. If a suitable $(h,v)$ does not
		exist, then $W_{i-1}$ is extended with the semi-guard $(t_i,s_i)$ to
        obtain $W_i$, see \cref{line:newguard}.
\end{enumerate}
\begin{algorithm}[ht]
\DontPrintSemicolon{}
\SetAlgoLined{}
\KwData{tree-based bipartite graph $G$ with $k=|G_V|$, and the $\prec$ order on
its vertices}
\KwResult{A set of semi-guards $W_i$ for $i=1,\ldots,k$}

\medskip

let ${(s_i)}_{i=1}^k$ be the elements of $G_V$ in decreasing $\prec$ order\;
let ${(t_i)}_{i=1}^k$ be a sequence of yet undefined elements of $G_H\cup \{\emptyset\}$\;
let ${(W_i)}_{i=1}^k$ be a sequence of yet undefined subsets of $G_H\times G_V$, let $W_0=\emptyset$\;

\medskip

\For{$i=1,\ldots,n$}
{
        $t_i\gets\max_{\prec}\{ t\in \Gamma(s_i) \mid t s_{i}\text{\ is not covered by
        }W_{i-1}\}$\tcp*{$\max_{\prec}\emptyset=\emptyset$}\label{line:ti}
        \uIf{$t_i\neq\emptyset$}
        {
            \uIf{$\exists (h,v)\in W_{i-1}$ such that $h\in \Gamma(s_i)$}
                {\label{line:ifextend}
                    let $(h,v)\in W_{i-1}$ be such that $h\in \Gamma(s_i)$, and
                    $h$ is max.\ in $\prec$ order\;\label{line:max}
                    $W_i\gets W_{i-1}-(h,v)+(t_i,v)$\;\label{line:modify}
                }
            \Else
            {
                $W_i\gets W_{i-1}+(t_i,s_i)$\;\label{line:newguard}
            }
        }
        \Else
        {
            $W_i\gets W_{i-1}$\;
        }
}
\caption{Finding a minimum cardinality set of semi-guards that cover the tree-based bipartite graph $G$}\label{alg:main}
\end{algorithm}
The specifics of implementing \Cref{alg:main} will be discussed in
\Cref{sec:efficient}, where we will prove that \Cref{alg:main} can be
implemented in quadratic time.

For every $i=1,\ldots,k$, both $W_i$ and $t_i$ are assigned a value exactly once
by \Cref{alg:main}, so they can be discussed as mathematical
variables (instead of thinking about them as variables that change their values as the
algorithm proceeds). Let us make a couple of simple observations about \Cref{alg:main}.

\begin{observation}\label{lemma:tisi}
     For any $i=1,\ldots,k$, $t_i\neq\emptyset$ if and only if $t_i\in G_H$ if and only if $t_i\in
	\Gamma(s_i)$.
\end{observation}
\begin{observation}\label{lemma:tisiGamma}
    If $t_i\in \Gamma(s_i)$, then $s_i\in \Gamma(t_i)$,
    $\min\Gamma(s_i)\le t_i$, and $\min\Gamma(t_i)\le s_i$.
\end{observation}
\begin{observation}\label{lemma:ti_eq_h}
    If $(h,v)\in W_i\setminus W_{i-1}$, then $h=t_i$.
\end{observation}
\begin{observation}\label{lemma:si_le_v}
    If $(h,v)\in W_i$, then $s_i\preceq v$.
\end{observation}
\begin{observation}\label{lemma:vvprime_distinct}
    If $(h,v),(h',v')\in W_i$ are distinct, then $v\neq v'$.
\end{observation}

\begin{lemma}\label{lemma:alg_semiguards}
    $W_i$ is a set of semi-guards for any $i=1,\ldots,k$. For any $(h,v)\in
    W_i$, there exists $j\le i$ such that $h=t_j$ and
    \begin{align*}
        \min\Gamma(v) &\preceq t_j\\
        \min\Gamma(v) &\in \Gamma(s_j).
    \end{align*}
\end{lemma}
\begin{proof}
    By induction on $i$. Note, that $W_0=\emptyset$, set during initialization.
    \begin{itemize}
        \item If $W_i=W_{i-1}$, then then the induction step is trivial.
        \item If $W_i=W_{i-1}+(t_i,s_i)$, then $(t_i,s_i)$ is a semi-guard by
            \Cref{lemma:semiguard_hv_edge}. Since $t_i\in \Gamma(s_i)$, we have
            $\min\Gamma(s_i)\le t_i$. Recall that $\prec$ extends $<$.
        \item If $W_i=W_{i-1}-(h,v)+(t_i,v)$: then $h\in \Gamma(s_i)$.
            Moreover, we apply the inductive hypothesis to
            $(h,v)\in W_{i-1}$. That is, there exists $j\le i-1$ such that
            $h=t_j$ and $\min\Gamma(v) \preceq t_j$ and $\min\Gamma(v) \in
            \Gamma(s_j)$. Also, $t_j\in \Gamma(s_i)$.

            Observe, that $\min\Gamma(v)\preceq t_j$ and $s_i\preceq s_j$. In
            addition, $\min\Gamma(v),t_j\in \Gamma(s_j)$ and $t_j\in
            \Gamma(s_i)$. In this case, \Cref{lemma:inheriting_neighbor} implies
            that $\min\Gamma(v)\in \Gamma(s_i)$.

            It follows that $s_i$ is a common neighbor of $t_i$ and $\min\Gamma(v)$.
            If $t_i\preceq t_j$, then by \Cref{lemma:semiguard_covers}, $(t_j,v)\in W_{i-1}$ covers $t_i s_i$,
            which is a contradiction. Therefore we must have $t_j\prec t_i$. By
            induction, $\min\Gamma(v)\preceq t_j$. By transitivity and
            \Cref{lemma:linear_order_intersecting_neighborhoods}, we have
            \begin{align*}
                &\min\Gamma(v)\preceq t_i,\\
                &\min\Gamma(t_i)\in \Gamma(\min\Gamma(v)),
            \end{align*}
            which proves that $(t_i,v)$ is a semi-guard. The inductive
            hypothesis holds for $(h,v)\in W_i$, since we already showed that $\min\Gamma(v)\in
            \Gamma(s_i)$.
    \end{itemize}
\end{proof}

\begin{lemma}\label{lemma:alg_order}
    If $W_i=W_{i-1}-(h,v)+(t_i,v)$, then $h\prec t_i$, $s_i\in\Gamma(h)$,
    $\min\Gamma(t_i)\in\Gamma(h)$, and $\min\Gamma(v)\in \Gamma(s_i)$.
\end{lemma}
\begin{proof}
    The relations $h\prec t_i$ and $\min\Gamma(v)\in \Gamma(s_i)$ are shown
    explicitly in the proof of \Cref{lemma:alg_semiguards}, and the algorithm
    requires $s_i\in \Gamma(h)$. Because $s_i\in \Gamma(t_i)$,
    \Cref{lemma:linear_order_intersecting_neighborhoods} implies
    $\min\Gamma(t_i)\in\Gamma(h)$.
\end{proof}

Note, that the vertical components of semi-guards are preserved.
\begin{corollary}\label{lemma:vertical_preserved}
    If $(h,v)\in W_j$, then for any $i\ge j$ (and $i\le k$) there $\exists m\in
    \mathbb{N}$ such that $(t_m,v)\in W_i$ and $h\preceq t_m$, where $j\le m\le
    i$.
\end{corollary}

From \Cref{lemma:semiguard_covers} it easily follows that $W_j$ covers $t s_j$
for any $t\in \Gamma(s_j)$.
\begin{lemma}\label{lemma:Wi_covers_si}
	Let $1\le j\le k$ and $t\in \Gamma(s_j)$. Then $W_j$ covers $t s_j$.
\end{lemma}
\begin{proof}
	We distinguish three cases.
    \begin{itemize}
        \item If $t_j=\emptyset$, then $W_{j-1}$ covers $t s_j$. By
            \Cref{alg:main}, we have $W_j=W_{j-1}$.

        \item If $t_j\neq\emptyset$ and $W_j=W_{j-1}-(h,v)+(t_j,v)$ and
            $t_j\prec t$, then $(h,v)\in W_{j-1}$ does not cover $t s_j$ (since
            $h\prec t$ by~\Cref{lemma:alg_order}), but $W_{j-1}$ does
            cover $t s_j$ (otherwise $t_j$ was not maximal in $\prec$ order);
            it follows that $W_j$ also covers $t s_j$;

        \item If $t_j\neq\emptyset$ and $W_j=W_{j-1}-(h,v)+(t_j,v)$ and
            $t\preceq t_j$, then by \Cref{lemma:semiguard_covers}, $(t_j,v)$
            covers $t s_j$, because $t_j\in \Gamma(s_j)$ by \Cref{lemma:tisi}.
    \end{itemize}
    Thus we have shown that $W_j$ covers any $t s_j\in E(G)$.
\end{proof}
It follows that if $\exists (h,v)\in W_{i-1}$ such that $h\in \Gamma(s_i)$, then
$h\neq t_i$, otherwise $t_i s_i$ would not be covered by $W_i$. We also need to
make sure that $ts_j\in E(G)$ remains covered by $W_i$ for any $i>j$.
\begin{lemma}\label{lemma:cover_preserved}
	Let $1\le j\le i\le k$ and $t\in \Gamma(s_j)$. Then $W_i$ covers $t s_j$.
	Furthermore, if $j<i$ and $W_{i-1}\setminus W_i$ covers $t s_j$, then
	$W_i\setminus W_{i-1}$ also covers $t s_j$.
\end{lemma}
\begin{proof}
	We prove the complete statement by induction on $i$.
	\Cref{lemma:Wi_covers_si} proves that the inductive
	hypothesis holds for $i=j$. Let $i>j$ and suppose that $(h,v)\in W_{i-1}$
	covers $t s_j$. If $(h,v)\in W_i$, then the induction step is trivial.

    \medskip

    Suppose from now on, that $W_i=W_{i-1}-(h,v)+(t_i,v)$. As $(h,v)$ is a
    semi-guard, we have $\min\Gamma(v)\in \Gamma(\min\Gamma(h))$.
    Because $(h,v)$ covers $t s_j$, we know that $t\preceq h$, $s_j\preceq v$,
    and $\min\Gamma(h)\in \Gamma(t)$, $\min\Gamma(v)\in \Gamma(s_j)$.
	Recall \Cref{lemma:alg_order}, and note that $\min\Gamma(v),h\in\Gamma(s_i)$
    and $t\preceq h\prec t_i$.
	We will show that $s_i\in \Gamma(t)$. We have two cases.
	\begin{itemize}
		\item If $t\preceq\min\Gamma(v)$: then we apply
			\Cref{lemma:inheriting_neighbor} as follows. We have
			$t\preceq\min\Gamma(v)$ and $s_i\preceq s_j$. Also,
			$\min\Gamma(v)\in\Gamma(s_i),\Gamma(s_j)$ and $t\in \Gamma(s_j)$.
			Thus $t\in \Gamma(s_i)$ follows.

		\item If $\min\Gamma(v)\preceq t\preceq h$:  Note, that
            $\min\Gamma(h)\preceq s_i\preceq s_j$. Observe, that
            \Cref{lemma:ell_gamma_convex} applies: indeed, $t,h\in
            \Gamma(\min\Gamma(h))$, $\min\Gamma(v),h\in \Gamma(s_i)$, and
            $\min\Gamma(v),t\in \Gamma(s_j)$. Therefore $t\in \Gamma(s_i)$ holds.
	\end{itemize}
	In any case, $s_i$ is a common neighbor of $t$ and $t_i$, therefore by
	\Cref{lemma:linear_order_intersecting_neighborhoods} we have $\min\Gamma(t_i)\in
    \Gamma(t)$. Since $t\prec t_i$, this concludes the proof that $(t_i,v)$
    covers $t s_j$.
\end{proof}

\begin{theorem}\label{thm:Wk_covers_G}
    The semi-guard set $W_k$ covers $E(G)$. Consequently, the set of $r$-guards
    \begin{equation}\label{eq:rguards_of_Wk}
            \left\{\min\Gamma(v)\min\Gamma(h) \Big.\mid (h,v)\in W_k\right\}
    \end{equation}
    also covers $E(G)$.
\end{theorem}
\begin{proof}
    \Cref{lemma:cover_preserved} proves that the
    semi-guard set $W_k$ covers $E(G)$. By~\Cref{def:semiguard}, anything covered
    by the semi-guard $(h,v)$ is covered by the $r$-guard
    $\min\Gamma(v)\min\Gamma(h)$.
\end{proof}

\section{Constructing a maximal independent set}\label{sec:indep}
In this section we will show that \Cref{alg:independent} selects a set $I_1$ of
pairwise $r$-independent edges (\Cref{def:independent}) of the tree-based
bipartite graph $G$ such that $|I_1|=|W_k|$.
\begin{algorithm}[ht]
\SetAlgoLined{}
\KwData{tree-based bipartite graph $G$, an $\prec$ order on its vertices, and
    the sequences ${(W_i)}_{i=1}^k$ and ${(t_i)}_{i=1}^k$ produced by \Cref{alg:main}}
\KwResult{A set $I_i$ of pairwise $r$-independent edges of $G$ for every
$i=1,\ldots,k$}

\medskip

let $I_{k+1}=\emptyset$\;
\For{$i=k,\ldots,1$}{\label{line:indep_for}
	\uIf{$t_i\neq\emptyset$ and $t_i s_i$ is $r$-independent from every edge in
    $I_{i+1}$}
    {\label{line:tisi_indep}
        \uIf{$W_i=W_{i-1}+(t_i,s_i)$}
        {
            let $I_i\gets I_{i+1}+t_i s_i$\;
        }
        \ElseIf{$W_i=W_{i-1}-(h,v)+(t_i,v)$}
        {
            \uIf{$\exists t_n s_n\in I_{i+1}$ such that $(h,v)$ covers $t_n s_n$}
            {
                let $I_i\gets  I_{i+1}+t_i s_i$\;
            }
            \uElseIf{$\exists t_n s_n\in I_{i+1}$ such that $h\in \Gamma(s_n)$ and
            $h\prec t_n$}
            {
                let $I_i\gets I_{i+1}+t_i s_i$\;
            }
            \Else
            {
                let $I_i\gets I_{i+1}$\;
            }
        }
	}
	\Else
    {
		let $I_i\gets I_{i+1}$\;
	}
}
\Return{$I_1$}\;
\caption{Finding a maximum size set of $r$-independent edges of $G$}\label{alg:independent}
\end{algorithm}
\begin{lemma}\label{lemma:ti_cover_preserved}
	Let $1\le j\le i\le k$. If $t_j\in G_H$, then $W_i$
	covers $t_j s\in E(G)$ for any $s\in \Gamma(t_j)$.
\end{lemma}
\begin{proof}
	If $s_j\prec s$, then by \Cref{lemma:Wi_covers_si,lemma:cover_preserved},
	$W_i$ covers $t_j s$.

	Suppose from now on that $s\preceq s_j$. Note, that $t_j\in \Gamma(s)\cap
	\Gamma(s_j)$. By \Cref{lemma:linear_order_intersecting_neighborhoods}, we
	get $\min\Gamma(s_j)\in \Gamma(s)$, thus $(t_j,s_j)\in W_j$ covers $t_j s$.
	If $i=j$, then the proof is complete.

	Suppose also, that $i>j$. By \Cref{lemma:vertical_preserved}, there exists
	some $(t_m,s_j)\in W_i$ where $j\le m\le i$. By
	\Cref{lemma:Wi_covers_si,lemma:cover_preserved}, $(t_m,s_j)$ covers $t_j
	s_j$. Therefore, $\min\Gamma(t_m)\in \Gamma(t_m)$ and $t_j\preceq t_m$. In
	addition, we have shown previously that $\min\Gamma(s_j)\in \Gamma(s)$,
	therefore $(t_m,s_j)$ covers $t_j s$.
\end{proof}

\begin{corollary}\label{lemma:ti_distinct}
	If $j<i$ and $t_j,t_i\in G_H$, then $t_j\neq t_i$.
\end{corollary}
\begin{proof}
	By \Cref{lemma:ti_cover_preserved}, $W_{i-1}$ covers $t_j s$ for any $s\in
	\Gamma(t_j)$. If the edge $t_j
	s_i\in E(G)$ exists, then substituting $s=s_i$ implies that $t_j s_i$ is
	also covered by $W_{i-1}$.
\end{proof}

\begin{corollary}\label{lemma:indices}
    If $(t_m,s_j)\in W_i$, then $1\le j\le m\le i$.
\end{corollary}
\begin{proof}
    Knowing \Cref{lemma:ti_distinct}, the statement follows easily by checking
    \Cref{alg:main}.
\end{proof}

\begin{lemma}\label{lemma:guards_distinct}
	If $(h,v),(h',v')\in W_i$ are distinct, then $h\neq h'$ and $v\neq v'$.
\end{lemma}
\begin{proof}
    By \Cref{lemma:vvprime_distinct}, we have $v\neq v'$.
    There exist two distinct integers $j,m\le i$ such that $(h,v)\in
    W_j\setminus W_{j-1}$ and $(h',v')\in W_m\setminus W_{m-1}$.
    From \Cref{lemma:ti_eq_h} it follows that $h=t_j$ and $h'=t_m$, therefore
    $h\neq h'$.
\end{proof}

Next, we show that \Cref{alg:independent} greedily selects a decreasing sequence
of $r$-independent edges of $G$. Notice, that the index $i$ is enumerated in
decreasing-order by the \textbf{for}-loop on \cref{line:indep_for}.

\begin{definition}\label{def:Ii}
    Let $(I_i)_{i=1}^{k+1}$ be a sequence of subsets of $E(G)$, defined by
    \Cref{alg:independent}.
\end{definition}
Notice, that $I_{i+1}\subseteq I_i$, therefore
\begin{equation}\label{eq:Idecreasing}
    I_1\supseteq I_2\supseteq I_3\supseteq \ldots \supseteq I_k\supseteq
    I_{k+1}=\emptyset.
\end{equation}

\begin{lemma}\label{lemma:Ii_elements}
    For any $i=1,\ldots,k$, we have $I_i\subseteq \{ t_m s_m \mid i\le m\le k,\ t_m\in G_H\}$.
\end{lemma}
\begin{proof}
    Trivial, since $I_i\setminus I_{i+1}\subseteq \{ t_i s_i \}$.
\end{proof}

\begin{lemma}\label{lemma:Ii_indep}
    For any $i=1,\ldots,k+1$, the elements of $I_i$ are pairwise
    $r$-independent.
\end{lemma}
\begin{proof}
    By induction on decreasing $i$. The statement trivially holds for
    $I_{k+1}=\emptyset$. By~\cref{line:tisi_indep}, $t_i s_i$ can be an element of $I_i$ only if $t_i
    s_i$ is $r$-independent from every edge in $I_{i+1}$.
\end{proof}

The next lemma is the basis for the lower bound on the cardinality of $I_1$. To
keep the exposition of the main argument of this section concise, the proof of
\Cref{lemma:independence_induction} is postponed to \Cref{sec:postponed_indep}.
\begin{lemma}\label{lemma:independence_induction}
    For any $i=1,\ldots,k$ and $(h,v)\in W_i$ at least one of the following holds:
    \begin{enumerate}
        \item $\exists t_n s_n\in I_{i+1}$ such that $(h,v)$ covers $t_n s_n$, or
        \item $\exists t_n s_n\in I_{i+1}$ such that $h\in \Gamma(s_n)$ and $h\prec t_n$, or
        \item there exists $q\in \mathbb{N}$ such that $1\le q\le i$, $t_q
            s_q\in I_1$, and $(t_q,v)\in W_q$.
    \end{enumerate}
\end{lemma}

\begin{lemma}\label{lemma:I1_ge_Wk}
    The cardinalities of $I_1$ and $W_k$ satisfy $|I_1|\ge |W_k|$.
\end{lemma}
\begin{proof}
    Let $i=k$ and notice that in this special case
    \Cref{lemma:independence_induction} states the following: for any $(h,v)\in
    W_k$ there exists $q_{(h,v)}$ such that $t_{q_{(h,v)}}s_{q_{(h,v)}}\in I_1$
    and $(t_{q_{(h,v)}},v)\in W_{q_{(h,v)}}$.

    By~\Cref{lemma:ti_distinct}, it is sufficient to show that for any two
    $(h,v),(h',v')\in W_k$, we have $q_{(h,v)}\neq q_{(h',v')}$.
    Suppose for a contradiction, that $q=q_{(h,v)}=q_{(h',v')}$. Then $t_q=h=h'$
    and $(t_q,v),(t_q,v')\in W_{q}$. By~\Cref{lemma:guards_distinct}, we must
    have $(h,v)=(h',v')$, which is a contradiction.
\end{proof}

In other words, the following min-max theorem holds.
\begin{theorem}\label{thm:independent}
    The minimum number of $r$-guards required to cover $G$ is equal to the maximum
    number of pairwise $r$-independent edges of $G$.
	In particular, $W_k$ is a
    minimum cardinality set of semi-guards covering $G$, and $I_1$ is a
    maximum size set of $r$-independent edges of $G$, such that $|W_k|=|I_1|$.
\end{theorem}
\begin{proof}
    By \Cref{lemma:indep}, the minimum number of $r$-guards covering $G$ is at
    least the maximum number of $r$-independent edges of $G$. Recall
    \Cref{thm:Wk_covers_G} and \Cref{lemma:Ii_indep}. The set of $r$-guards
    on~\eqref{eq:rguards_of_Wk} covers $E(G)$, and the elements of $I_1$ are
    pairwise $r$-independent. Hence, we must have $|I_1|\le |W_k|$.

    Finally, by~\Cref{lemma:I1_ge_Wk}, the equality $|I_1|=|W_k|$ must hold. In
    other words, $I_1$ and $W_k$ are witnesses to the min-max equality.
\end{proof}

\section{Constructing the cover in quadratic time}\label{sec:efficient}
Let $n=|G_H|+|G_V|$ be the number of vertices of $G$. To efficiently work with
the graph $G$, we do not need to store the complete neighborhood $\Gamma(s)$ for
every $s\in G_H\cup G_V$. If the $R$-trees $T_H$ and $T_V$ are known, then it is
sufficient to store the endpoints $\partial \Gamma(s)$ of the path induced by
$\Gamma(s)$ in the corresponding $R$-tree for every $s\in G_H\cup G_V$.
\begin{definition}[Sparse representation of a tree-based bipartite graph]
    The sparse representation of the tree-based bipartite graph $G$ is the
    triplet $(T_H,T_V,\partial\Gamma)$ composed of its two $R$-trees and the
    function $\partial\Gamma: s\mapsto \partial\Gamma(s)$ mapping $G_H\cup G_V$
    to its at most two element subsets.
\end{definition}
Notice, that even if $G$ has $\Omega(n^2)$ edges, both the $R$-trees and
$\partial\Gamma$ can be stored in $\mathcal{O}(n)$ space.

Let us quickly show that computing the raster graph $G$ of a simple orthogonal
polygon $P$ takes linear time.
\begin{theorem}[\citet{gyori_generalized_1996}]\label{lemma:Rtree_contruction}
    For any given simple orthogonal polygon $P$, its $R$-trees $T_H$ and $T_V$
    can be determined in linear time.
\end{theorem}
\begin{lemma}\label{lemma:sparse_representation}
    The sparse representation of the raster graph $G$ of a simple orthogonal
    polygon $P$ can be computed in linear time.
\end{lemma}
\begin{proof}[Sketch of the proof.]
    The lemma is evident once one studies the proof of
    \Cref{lemma:Rtree_contruction}. During the construction of $T_H$ and $T_V$,
    one may link the segments of $P$ to the slices containing the segment.
    For any $h\in G_H$, take the one-one vertical segment of $P$ containing the vertical
    sides of the rectangle $h$. Clearly, the one or two vertical slices
    containing these vertical segments form the set $\partial \Gamma(h)$. One may proceed analogously for
    any $v\in G_V$.
\end{proof}
For a full technical description, the diligent reader is referred
to~\cite[Chapter~4 and Appendix~B]{mezei_extremal_2017}. From now on, we are
back to working on a general tree-based bipartite graph $G$.

\begin{theorem}[\citet{gabow_linear-time_1985}]\label{thm:NCA}
    The nearest common ancestors of $m$ pairs of vertices on an $n$-vertex tree
    can be determined in $\mathcal{O}(n+m)$ time.
\end{theorem}
\begin{definition}
    Given $S\subseteq G_H$ or $S\subseteq G_V$, let $\NCA(S)$ be the nearest
    common ancestor of $S$ in $T_H$ or $T_V$, respectively.
\end{definition}

NCA queries are essential to efficiently computing on the sparse
representation of a tree-based bipartite graph. For example, the next
observation shows that one NCA query is sufficient to answer whether two
vertices of $G$ are less than or equal to one another.
\begin{observation}\label{lemma:leNCA}
    Given $a,b\in G_H$ or $a,b\in G_V$, we have $a\le b$ if and only if
    $\NCA(a,b)=a$.
\end{observation}
\begin{lemma}\label{lemma:minGammaNCA}
    Given $s\in G_H\cup G_V$, we have $\min\Gamma(s)=\NCA(\partial
    \Gamma(s))$.
\end{lemma}
\begin{proof}
    Trivial, since $\Gamma(s)$ induces a path in the appropriate $R$-tree.
\end{proof}

Recall \Cref{def:partial}. The next lemma implies that the intersection of two
neighborhoods in $G$ can be computed efficiently.
\begin{lemma}\label{lemma:intersectNCA}
    Given $s_{11},s_{12},s_{21},s_{22}\in G_H$ or
    $s_{11},s_{12},s_{21},s_{22}\in G_V$, \Cref{alg:intersection} determines the
    set $\partial ([s_{11},s_{12}]\cap [s_{21},s_{22}])$ via 7 NCA queries.
\end{lemma}
\begin{proof}
	Note, that $U=[s_{11},s_{12}]\cap [s_{21},s_{22}]$ also induces a path (or
	an empty graph) in the appropriate $R$-tree. By definition, $\partial
	U\subseteq [s_{11},s_{12}]\cap [s_{21},s_{22}]$, so the elements of
	$\partial U$ are common ancestors of a subset of
	$\{s_{11},s_{12},s_{21},s_{22}\}$. In fact, elements of $\partial U$ are
	nearest common ancestors of subsets of $\{s_{11},s_{12},s_{21},s_{22}\}$,
	because a nearer common ancestor must still be contained in both
	$[s_{11},s_{12}]$ and $[s_{21},s_{22}]$. The proof can be completed by a
	simple, but slightly tedious case analysis.
\end{proof}
\Cref{lemma:intersectNCA} also implies that $st\in
E(G)$ can be checked with 7 NCA queries since it holds if and only if
$\{s\}\cap \Gamma(t)\neq\emptyset$, which is equivalent to $\partial([s,s]\cap
\Gamma(t))\neq\emptyset$. In particular, this means that $r$-visibility
(\Cref{def:rvisionedge}) can also be checked efficiently.

\begin{algorithm}[ht]
	\SetAlgoLined
	$s_1\gets\NCA(s_{11},s_{12})$\;
	$s_2\gets\NCA(s_{21},s_{22})$\;
	$s_0\gets \NCA(s_{1},s_{2})$\;
	\If{$s_0\in \{s_1,s_2\}$}{
		$S\gets \{\NCA(s_{1i},s_{2j})\ |\ i,j\in \{1,2\} \}\setminus
		\{s_0\}$\;
		\uIf{$s_1=s_2$ and $|S|\le 1$}{
			\Return{$S\cup \{s_0\}$}\tcp*{$s_0=s_1=s_2$ is in the
	intersection}
		}
		\Else{
			\Return{$S$}\;
		}
	}
	\Return $\emptyset$\tcp*{The two paths do not intersect each other}
	\caption{Computing $\partial \big({[s_{11},s_{12}]}\cap
	{[s_{21},s_{22}]}\big)$, i.e., determining the endpoints of the intersection of two paths
from their respective endpoints in the $R$-tree}\label{alg:intersection}
\end{algorithm}

\begin{lemma}\label{lemma:linorder_BFS}
    A linear order $\prec$ on $G_H$ and $G_V$ (which is compatible with $<_b$, see
    \Cref{def:totalorder}) can be determined in $\mathcal{O}(n)$ time.
\end{lemma}
\begin{proof}
    Recall \Cref{def:unifiedorder,def:totalorder} and \Cref{lemma:minGammaNCA}.
    By~\Cref{thm:NCA} and \Cref{lemma:minGammaNCA}, $\min\Gamma(s)$ can be
    computed in $\mathcal{O}(n)$ time for every $s\in G_H\cup G_V$
    simultaneously. Subsequently, the linear orders can be determined with simple
    breadth-first searches on the $R$-trees.
\end{proof}

It is now easy to see that \Cref{alg:main} can be emulated in $\mathcal{O}(n^3)$
time. The cardinality of the set of all possible NCA queries is $\mathcal{O}(n^2)$,
which can be computed ahead of running \Cref{alg:main}. The elements of
$\Gamma(s_i)$ can be easily listed in $\prec$ order. It only remains to
describe how to check for each $t\in
\Gamma(s_i)$ whether some $(h,v)\in W_{i-1}$ covers
$ts_i$ or not. Recall \Cref{lemma:alter_def_semiguard}. Checking whether $t\preceq h$
and $s_i\preceq v$ is trivial once the linear
order is computed. Deciding whether both $\Gamma(h)\cap\Gamma(t)$ and
$\Gamma(h)\cap\Gamma(t)$ are non-empty requires $\mathcal{O}(1)$ NCA queries by
\Cref{lemma:intersectNCA}. If $t_i\neq\emptyset$,
it is trivial to find in $\mathcal{O}(n)$ every $(h,v)\in W_{i-1}$
such that $h\in \Gamma(s_i)$ and $h\prec t_i$.

This naive approach can be improved as follows.
\begin{theorem}\label{thm:running_time}
    \Cref{alg:main} can be emulated in $\mathcal{O}(n^2)$ time (if the tree-based
    bipartite graph $G$ is encoded with sparse representation).
\end{theorem}
\begin{proof}
    To achieve an $\mathcal{O}(n^2)$ running time, it is sufficient to augment
    the above argument by a method that computes $t_i$ in $\mathcal{O}(n)$. Recall,
    that the set of all possible NCA queries can be computed in
    $\mathcal{O}(n^2)$ in a preprocessing phase.
    For each $t\in \Gamma(s_i)$, we will compute the number of semi-guards
    $(h,v)\in W_{i-1}$ that cover $t s_i$.

    For every $h\in G_H$, let
    \begin{equation*}
        C_h:=\{ t\in \Gamma(\min\Gamma(h))\mid t\preceq h\}.
    \end{equation*}
    Notice, that $C_h$ induces a path in $T_H$ (because $\prec$ is compatible
    with $<$), and $\partial C_h$ can be determined in
    $\mathcal{O}(n)$ time for each $h\in G_H$.

    Suppose $(h,v)\in W_{i-1}$. Define
    \begin{equation*}
        C_{(h,v)}(s_i):=\{ t\in \Gamma(s_i) \mid (h,v)\text{\ covers\ }ts_i\}.
    \end{equation*}
    Recall \Cref{def:semiguard}. If some $(h,v)$ covers a $t s_i\in E(G)$, then
    $\min\Gamma(v)\in \Gamma(s_i)$. Thus, if $\min\Gamma(v)\notin \Gamma(s_i)$,
    then $C_{(h,v)}(s_i)=\emptyset$. However, if $\min\Gamma(v)\in \Gamma(s_i)$,
    then 
    \begin{align*}
        C_{(h,v)}(s_i)&=\{ t\in \Gamma(s_i)\cap\Gamma(\min\Gamma(h))\mid t\preceq h\}=\\
                                                           &=\Gamma(s_i)\cap \{
                                                           t\in
                                                       \Gamma(\min\Gamma(h))\mid
        t\preceq h\}=\\ &=\Gamma(s_i)\cap C_h.
    \end{align*}
    For each $s_i\in G_V$, one can collect the semi-guards $(h,v)\in
    W_{i-1}$ such that $\min\Gamma(v)\in \Gamma(s_i)$ in $\mathcal{O}(n)$ time.
    For such semi-guards, by \Cref{lemma:intersectNCA}, we can determine $\partial C_{(h,v)}(s_i)$
    from $\partial \Gamma(s_i)$ and $\partial C_h$ via a constant number of NCA
    queries. Because $|W_i|\le k$, at most $\mathcal{O}(n)$ queries are required
    to determine every $C_{(h,v)}(s_i)$ such that $(h,v)\in W_i$.

    By a simple traversal of $\Gamma(s_i)$ as a \emph{path} in $T_H$ (starting
    and ending at elements of $\partial\Gamma(s_i)$) one can count the number of
    semi-guards $(h,v)\in W_{i-1}$ such that $t\in C_{(h,v)}(s_i)$. Indeed, increment
    the counter the first time an element of $\partial C_{(h,v)}(s_i)$ is
    encountered, and decrement the counter when the other element of $\partial
    C_{(h,v)}(s_i)$ is reached by the traversal.

    If the counter does not become zero at any $t\in
    \Gamma(s_i)$, then $t_i=\emptyset$. Otherwise, $t_i$ is the largest $t\in
    \Gamma(s_i)$ in $\prec$ order where the counter becomes zero.
\end{proof}

As discussed in the introduction, log-linear (and linear) algorithms are available
for class-3 (and class-2) simple orthogonal polygons. It is an open question
whether such a low-complexity algorithm exists for class-4 polygons as well.
\begin{conjecture}
    \Cref{alg:main} can be emulated in $\mathcal{O}(n\log n)$ time (if the tree-based
    bipartite graph $G$ is encoded with sparse representation).
\end{conjecture}

\section{Computing the vertex list of an \texorpdfstring{\boldmath$r$}{r}-star}\label{sec:raster_quadratic}
We are ready to show that the $r$-star covered by an $r$-guard
can be determined in $\mathcal{O}(n)$ time.
\begin{proof}[Proof of \Cref{thm:alg_typeA}]
    Recall \Cref{lemma:equivalence}.
	The sparse representation $(T_H,T_V,\partial\Gamma)$ of the raster graph
	$G$ of the simple orthogonal polygon $P$ can be determined in linear time
	according to \Cref{lemma:sparse_representation}. Subsequently, \Cref{alg:main} can
    report a minimum cardinality set of covering $r$-guards in
    $\mathcal{O}(n^2)$ time. 

    Let $(h,v)\in W_k$ be a fixed semi-guard; the set of edges covered by the
    $r$-guard $\min\Gamma(v)\min\Gamma(h)\in E(G)$ contains the edges covered by
    $(h,v)$ in $G$. By \Cref{def:rguard}, the set of edges covered by
    $\min\Gamma(v)\min\Gamma(h)$ is equal to
    \begin{equation*}
        \{ h'v'\in E(G) \mid h'v'\boxdot \min\Gamma(v)\min\Gamma(h)\}
    =\bigcup_{\substack{h'v'\in E(G) \\ h'\in \Gamma(\min\Gamma(h))\\ v'\in
        \Gamma(\min\Gamma(v))}} \{ h'v'\}=\bigcup_{\substack{v'\in
        \Gamma(\min\Gamma(v))\\ h'\in\Gamma(v')\cap
\Gamma(\min\Gamma(h))}} \{h'v'\}.
    \end{equation*}
    In the orthogonal polygon, the $r$-star covered by the pixels of these edges is
    equal to
    \begin{equation*}
        \{ h'\cap v' \mid h'v'\in E(G),\ h'v'\boxdot \min\Gamma(v)\min\Gamma(h)\}=
        \bigcup_{\substack{v'\in
        \Gamma(\min\Gamma(v))\\ h'\in\Gamma(v')\cap
        \Gamma(\min\Gamma(h))}} h'\cap v'
    \end{equation*}
    Let us take the boundary of both expressions of the last equation. Because
    the region is simply connected, the boundary is closed and connected.
    Applying \Cref{lemma:boundary} twice, we
    get
    \begin{align}
        \partial\{ h'\cap v' \mid h'v'\in E(G),\ h'v'\boxdot \min\Gamma(v)\min\Gamma(h)\}=
        \closure\left(\SymmDiff_{\substack{v'\in
        \Gamma(\min\Gamma(v))\\ h'\in\Gamma(v')\cap
        \Gamma(\min\Gamma(h))}} \partial(h'\cap v')\right)\nonumber\\
        =\closure\left(\SymmDiff_{v'\in
                \Gamma(\min\Gamma(v))}\left(\SymmDiff_{\substack{h'\in\Gamma(v')\cap
        \Gamma(\min\Gamma(h))}} \partial(h'\cap v')\right)\right)\nonumber\\
        =\closure\left(\SymmDiff_{v'\in
                \Gamma(\min\Gamma(v))}\partial\left(\bigcup_{\substack{h'\in\Gamma(v')\cap
        \Gamma(\min\Gamma(h))}}h'\cap
v'\right)\right).\label{eq:rstar_boundary}
    \end{align}
    Note that by \Cref{lemma:boundary}, the closure operator only adds
    finitely many points to the right hand expressions. 
    Let us denote
    \begin{equation*}
        R_h(v'):=\bigcup_{\substack{h'\in\Gamma(v')\cap
        \Gamma(\min\Gamma(h))}}h'\cap v'.
    \end{equation*}
    For a visualization of $R_h(v')$, see \Cref{fig:rstar}.
    \begin{figure}[ht]
    	\centering
    	\begin{tikzpicture}
            \fill[gray,opacity=0.1] (-1.5,5)--(-1.5,2.5)--(-2,2.5)--(-2,2)--(0,2)--(0,-1)--(1,-1)--(1,-2)--(2.5,-2)--(2.5,0)--(3.5,0)--(3.5,-1)--(4.5,-1)
                --(4.5,-0.5)--(8,-0.5)--(8,1.5)--(7,1.5)--(7,0.5)--(6,0.5)--(6,4.5)--(4,4.5)--(4,6.5)
                -- (1.5,6)--(1.5,3.5)--(0,3.5)--(0,5.5);
    
            \fill[red, fill opacity=0.2] (-2,2) rectangle (6,2.5);
            \fill[blue,fill opacity=0.2] (1,-2) rectangle (1.5,3.5);
    
            \fill[green,fill opacity=0.2] (2.5,0) rectangle (6,3.5);
            \fill[green,fill opacity=0.2] (1,-2) rectangle (2.5,3.5);
            \fill[green,fill opacity=0.2] (0,-1) rectangle (1,3.5);
            \fill[green,fill opacity=0.2] (-1.5,2) rectangle (0,3.5);
            \fill[green,fill opacity=0.2] (-2,2) rectangle (-1.5,2.5);
    
            \draw[dashed] (-1.5,2.0)--(-1.5,2.5);
            \draw[dashed] (0,2)--(0,3.5);
            \draw[dashed] (2.5,0)--(2.5,6.2);
            \draw[dashed] (3.5,0)--(3.5,6.4);
            \draw[dashed] (1.0,-1)--(2.5,-1);
            \draw[dashed] (4,-1)--(4,4.5);
            \draw[dashed] (4.5,-0.5)--(4.5,0);
            \draw[dashed] (4.5,3.5)--(4.5,4.5);
            \draw[dashed] (6,-0.5)--(6,0);
            \draw[dashed] (7,-0.5)--(7,0.5);
    
            \draw[dotted] (-1.5,5)--(-1.5,4);
            \draw[dotted] (0,5.5)--(0,4.5);
            \draw[dotted] (1.5,5)--(1.5,6);
            \draw (1.5,5)--(1.5,3.5)--(0,3.5)--(0,4.5);
            \draw (-1.5,4)--(-1.5,2.5)--(-2,2.5)--(-2,2)--(0,2)--(0,-1)--(1,-1)--(1,-2)--(2.5,-2)--(2.5,0)--(3.5,0)--(3.5,-1)--(4.5,-1)
                --(4.5,-0.5)--(8,-0.5)--(8,1.5)--(7,1.5)--(7,0.5)--(6,0.5)--(6,4.5)--(4,4.5)--(4,5.5)
                edge[dotted] (4,6.5);
            \draw[thick,dashed,red] (-2,2) rectangle (6,2.5);
            \draw[thick,dashed,blue] (1,-2) rectangle (1.5,3.5);
    
            \draw[thick,green] (4.5,0) rectangle (6,3.5);
    
            \node[anchor=center] at (5.25,5) {\Large $P$};
            \node[anchor=center] at (5.25,4) {$v'$};
            \node[anchor=center] at (5.25,1) {$R_h(v')$};
            \node[anchor=center] at (7,2.25) {$\min\Gamma(v)$};
            \node[anchor=center] at (-1.75,2.25) {$v$};
            \node[anchor=south] at (1.25,-2.75) {$\min\Gamma(h)$};
            \node[anchor=south] at (2,-1.75) {$h$};
    	\end{tikzpicture}
        \caption{The simple orthogonal polygon $P$ is drawn with solid segments, and
            the region bounded by it is filled with gray. 
            The slices $\min\Gamma(v)$ and $\min\Gamma(h)$ are filled with red and
            blue, respectively, while their boundaries are contoured with dashed
            lines. The rest of the vertical cuts are drawn with dashed lines, and
            most of the horizontal cuts are not shown to avoid clutter. The pixels
            of the edges covered by the edge $\min\Gamma(v)\min\Gamma(h)$ are tinted
            green. The part of $v'$ thus covered is contoured with thick green lines,
            which highlights the boundary of $R_h(v')$.
        }\label{fig:rstar}
    \end{figure}
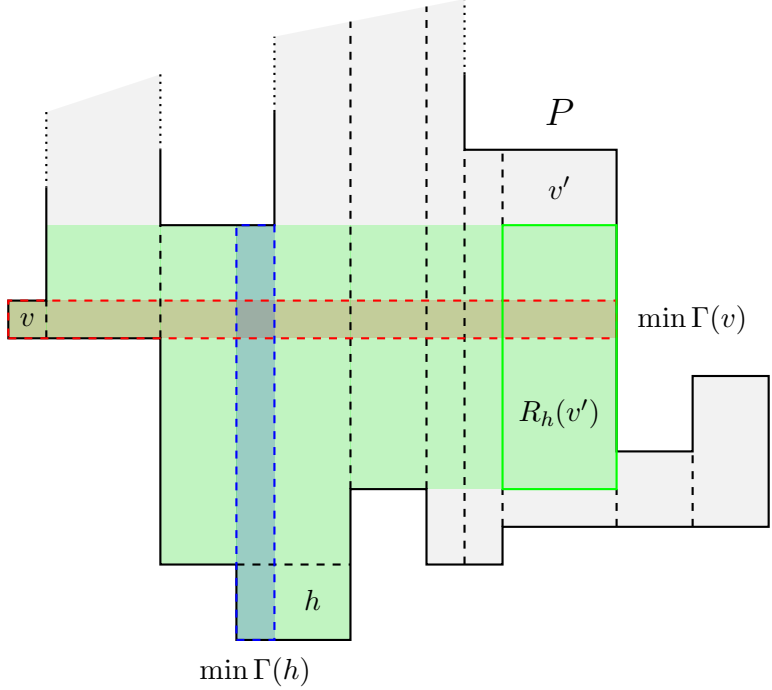
    By \Cref{lemma:convexHull}, we may write
    \begin{align}
        R_h(v')&=\mathrm{convexHull}\left(\bigcup_{\substack{h'\in\partial(\Gamma(v')\cap
        \Gamma(\min\Gamma(h)))}}h'\cap v'\right)\nonumber\\
        R_h(v')&=\mathrm{convexHull}\left(\bigcup_{\substack{h'\in\partial(\Gamma(v')\cap
        \Gamma(\min\Gamma(h)))}}\partial(h'\cap v')\right)\nonumber\\
                    \partial R_h(v')&=\partial\mathrm{convexHull}\left(\bigcup_{\substack{h'\in\partial(\Gamma(v')\cap
                    \Gamma(\min\Gamma(h)))}}\partial(h'\cap
                v')\right).\label{eq:partialRhv}
    \end{align}
	By \Cref{lemma:intersectNCA}, $\partial(\Gamma(v')\cap \Gamma(\min\Gamma(h)))$ can
    be computed with $\mathcal{O}(1)$ NCA queries. For a fixed $v'$, the
    boundary $\partial R_h(v')$ can be
    computed in $\mathcal{O}(1)$ time, because it is easily derived from
    the boundaries of two pixels. As $v'$ iterates
    through $\Gamma(\min\Gamma(v))$, in total $\mathcal{O}(n)$ NCA queries are
    sufficient to determine every $\partial R_h(v')$.

    Continuing~\cref{eq:rstar_boundary}, the boundary of the $r$-star is
    \begin{equation*}\label{eq:rstar_final}
        \partial\{ h'\cap v' \mid h'v'\in E(G),\ h'v'\boxdot \min\Gamma(v)\min\Gamma(h)\}=
        \closure\left(\SymmDiff_{v'\in \Gamma(\min\Gamma(v))} \partial R_h(v')\right).
    \end{equation*}
    Let $v',v''\in\Gamma(\min\Gamma(v))$ be two arbitrary elements. Then
    $R_h(v')$ and $R_h(v'')$ intersect if and
    only if $v'v''\in E(T_V)$ (by definition). The boundaries of $R_h(v')$ and
    $R_h(v'')$ intersect in a line
	segment or an empty set. Since $\Gamma(\min\Gamma(v))$ is the vertex set of
    a path in $T_V$, the orthogonal polygon bounding the $r$-star can be assembled in
	$\mathcal{O}(n)$ time for each $(h,v)\in W_k$.
\end{proof}

\renewcommand*{\bibfont}{\small}
{\sloppy\printbibliography{}}

\appendix

\section{Proofs postponed from
Section~\ref{sec:refined}}\label{sec:postponed_refined}
\begin{proof}[Proof of \Cref{lemma:inheriting_neighbor}]
    We have three cases.
    \begin{itemize}
        \item If $h_1$ and $h_2$ are comparable: since $h_1\preceq h_2$, we
            must have $h_1\le h_2$. As $h_2$ is a common neighbor of $v_1$ and $v_2$, by \Cref{lemma:linear_order_min_Gamma}, we have
            \begin{equation*}
                \min\Gamma(v_1)\le\min\Gamma(v_2)\le h_1\le h_2.
            \end{equation*}
            Since both $h_2,\min\Gamma(v_1)\in \Gamma(v_1)$, we have
            \begin{equation*}
                h_1\in [\min\Gamma(v_1), h_2]\subseteq
                \Gamma(v_1),
            \end{equation*}
            which is what we wanted to show.
        \item If $h_1$ and $h_2$ are not comparable, and
            $\min\Gamma(h_2)=v_\mathrm{root}$: we have $v_1,v_2\in
            \Gamma(h_2)$. By \Cref{lemma:leaf_comparable}, we know that
            $v_1$ and $v_2$ are comparable. Because $v_1\preceq
            v_2$, it follows that $v_1\le v_2$. Also, $v_1\in\Gamma(h_2)$
            implies
            \begin{equation*}
                \min\Gamma(h_2)\le v_1\le v_2.
            \end{equation*}
            We assumed that $v_2\in \Gamma(h_1)$.
            Because $v_2$ is a common neighbor of $h_1$ and $h_2$,
            \Cref{lemma:linear_order_intersecting_neighborhoods} implies
            $\min\Gamma(h_2)\in \Gamma(h_1)$. Therefore
            \begin{equation*}
                v_1\in [\min\Gamma(h_2), v_2]\subseteq
                \Gamma(h_1),
            \end{equation*}
            which implies $h_1\in \Gamma(v_1)$.

        \item If $h_1$ and $h_2$ are not comparable, and
            $\min\Gamma(h_2)\neq v_\mathrm{root}$:
            because $v_2$ is a common neighbor of $h_1$ and $h_2$,
            \Cref{lemma:linear_order_intersecting_neighborhoods} implies
            $\min\Gamma(h_2)\in \Gamma(h_1)$.
            Notice, that $h_1$ and $h_2$ are common neighbors of both
            $v_2$ and $\min\Gamma(h_2)$. Because the neighborhoods of
            both $v_2$ and $\min\Gamma(h_2)$ induce paths in $T_H$.
            By \Cref{lemma:minsarecomparable}, $\min\Gamma(v_2)$ and
            $\min\Gamma\Gamma(h_2)$ are comparable. However, due to
            \Cref{lemma:minless_comparable}, there cannot be strict inequality
            between them, therefore
            \begin{equation*}
                \min\Gamma\Gamma(h_2)=\min\Gamma(v_2)\le h_1,h_2.
            \end{equation*}
            Because $h_2$ is a common neighbor of $v_1$ and $v_2$,
            by \Cref{lemma:linear_order_intersecting_neighborhoods} we
            have $\min\Gamma(v_1)\le\min\Gamma(v_2)$. In addition,
            $h_2\in\Gamma(v_1)$ is equivalent to $v_1\in\Gamma(h_2)$,
            which implies $\min\Gamma(h_2)\le v_1$. It follows that
            \begin{equation*}
                \min\Gamma\Gamma(h_2)\le \min\Gamma(v_1)\le
                \min\Gamma(v_2),
            \end{equation*}
            so we must have equality:
            \begin{equation*}
                \min\Gamma(v_1)=\min\Gamma(v_2)=\min\Gamma\Gamma(h_2).
            \end{equation*}
            It follows that $v_1,v_2\in \Gamma(\min\Gamma\Gamma(h_2))$.

            Note, that $v_1,v_2\in \Gamma(h_2)$ implies
            $\min\Gamma(h_2)\le v_1,v_2$.
            Since $\min\Gamma(h_2)\neq v_\mathrm{root}$, by
            \Cref{lemma:minGGcloser} we have
            \begin{equation*}
                \min\Gamma\Gamma\Gamma(h_2)<\min\Gamma(h_2)\le v_1,v_2.
            \end{equation*}
            We apply \Cref{lemma:helly2} with
            $s_0=\min\Gamma\Gamma(h_2)$. It follows that $v_1$ and $v_2$
            are comparable, therefore
            \begin{equation*}
                \min\Gamma(h_2)\le v_1\le v_2.
            \end{equation*}
            As in the previous case, it follows that $h_1\in
            \Gamma(v_1)$.
    \end{itemize}
\end{proof}

\begin{proof}[Proof of \Cref{lemma:ell_gamma_convex}]
    Notice, that $h_1,h_2\in \Gamma(v_3)$ and $h_2\in \Gamma(v_1)$; by
    \Cref{lemma:inheriting_neighbor}, it follows that $h_1\in \Gamma(v_1)$.
    In other words, $h_1,h_2,h_3\in \Gamma(v_1)$ and $v_1,v_2,v_3\in
    \Gamma(h_1)$. By \Cref{lemma:linear_order_intersecting_neighborhoods}, we have
	\begin{equation*}
		\min\Gamma(h_1)\le\min\Gamma(h_2)\le\min\Gamma(h_3),\\
		\min\Gamma(v_1)\le\min\Gamma(v_2)\le\min\Gamma(v_3).
	\end{equation*}
    Note, that $\min\Gamma(h_2)\le v_3$ and $\min\Gamma(h_3)\le v_2$.
    Furthermore, $\min\Gamma(h_2),\min\Gamma(h_3),v_2,v_3\in \Gamma(h_1)$.
    Similarly, $\min\Gamma(v_2)\le h_3$ and $\min\Gamma(v_3)\le h_2$.

    \begin{itemize}
        \item If $v_2$ and $v_3$ are comparable: because $v_2\preceq
            v_3$, we have $v_2\le v_3$.

            If $v_1$ is comparable to $v_2$, then $v_1\le v_2\le v_3$. The last
            inequality implies $v_2\in [v_1,v_3]\subseteq \Gamma(h_2)$.

            Suppose $v_1$ is not comparable to $v_2$. Then
            \Cref{lemma:incomparable_common_neighbors} implies
            $\min\Gamma(h_1)=\min\Gamma(h_2)$.
            Because $v_2\in \Gamma(h_1)$, we have
            \begin{equation*}
                \min\Gamma(h_1)\le v_2\le v_3.
            \end{equation*}
            It follows that
            \begin{equation*}
                v_2\in [\min\Gamma(h_1),v_3]=[\min\Gamma(h_2),v_3]\subseteq
                \Gamma(h_2),
            \end{equation*}
            which implies that $h_2\in \Gamma(v_2)$.
        \item If $h_2$ and $h_3$ are comparable, then the previous reasoning
            holds by symmetry, so we can conclude that $h_2\in \Gamma(v_2)$,
            which is what we want.
    \end{itemize}

    Suppose from now on that $h_2,h_3$ are not comparable and $v_2,v_3$ are not
    comparable either.
    Notice, that \Cref{lemma:incomparable_common_neighbors} implies that
    $\min\Gamma(h_1)=\min\Gamma(h_2)$ and $\min\Gamma(v_1)=\min\Gamma(v_2)$.
    It follows that $h_2\in\min\Gamma(h_1)$.

    Since $h_1\in\Gamma(v_1)$ and $v_1\in \Gamma(h_2)$, we also have
    $\min\Gamma(h_1)\le v_1$ and $\min\Gamma(v_1)\le h_2$. As $v_1\in
    \Gamma(\min\Gamma(v_1))$, and $h_1\in\Gamma(\min\Gamma(h_1))$,
    \Cref{lemma:inheriting_neighbor} implies that
    $\min\Gamma(v_1)\in\Gamma(\min\Gamma(h_1))$. Also, $\min\Gamma\Gamma(h_1)\le
    \min\Gamma(v_1)$ due to \Cref{cor:minGG2}.

    All in all, $h_1,h_2,\min\Gamma(v_1)\in\Gamma(\min\Gamma(h_1))$. As
    $h_2,h_3\in \Gamma(v_1)$, we have
    \begin{equation*}
        \min\Gamma\Gamma(h_1)\le \min\Gamma(v_1)\le h_2,h_3.
    \end{equation*}
    Now \Cref{lemma:helly2} implies that the first inequality cannot be
    strict, thus
    \begin{equation*}
        \min\Gamma\Gamma(h_1)=\min\Gamma(v_1).
    \end{equation*}
    By symmetry, we also obtain that
    \begin{equation*}
        \min\Gamma\Gamma(v_1)=\min\Gamma(h_1).
    \end{equation*}
    Putting the two equations together, we obtain
    \begin{equation*}
        \min\Gamma(h_1)=\min\Gamma\Gamma(v_1)=\min\Gamma(\min\Gamma(v_1))=\min\Gamma(\min\Gamma\Gamma(h_1))=\min\Gamma\Gamma\Gamma(h_1).
    \end{equation*}
    By \Cref{lemma:minGGcloser}, this is possible only if
    $\min\Gamma(h_1)=v_\mathrm{root}$. Then \Cref{lemma:leaf_comparable} implies
    that $v_2,v_3\in \Gamma(h_1)$ are pairwise comparable, which is a
    contradiction.
\end{proof}

\section{Proof postponed from
Section~\ref{sec:indep}}\label{sec:postponed_indep}
The following argument is a quite complicated proof by induction. It is easy
to get disoriented by the main \textbf{for}-loops of
\Cref{alg:main,alg:independent} running in increasing and decreasing order of $i$,
respectively.
\begin{proof}[Proof of \Cref{lemma:independence_induction}]
    By induction on increasing $i$. We distinguish four main cases. The first
    three cases are relatively straightforward to resolve, while the last case
    leads to a far more complicated analysis with several nested subcases.

    \paragraph{Case 1: \boldmath$W_{i}=W_{i-1}$.} Then $t_i=\emptyset$, so
    $I_i=I_{i+1}$ and the claim follows by induction.

    \paragraph{Case 2: \boldmath$W_i=W_{i-1}+(t_i,s_i)$.} Note, that this is the case when
    $i=1$. We have two sub-cases.
    \begin{itemize}
        \item If $\exists t_n s_n\in I_{i+1}$ that is $r$-dependent on $t_i
            s_i$: by \Cref{alg:independent}, we have $I_{i+1}=I_{i}$. By
            \Cref{lemma:Ii_elements}, we have $n>i$, therefore $s_n\prec s_i$.
            \begin{itemize}
                \item If $t_n\preceq t_i$: then by
                    \Cref{lemma:alter_def_semiguard}, $(t_i,s_i)$ covers $t_n s_n$.
                \item If $t_i\prec t_n$: then by $r$-dependence,
                    $\min\Gamma(s_i)\in \Gamma(s_n)$ and
                    $\min\Gamma(t_n)\in \Gamma(t_i)$. Note, that
                    $\min\Gamma(s_i)\le t_i\preceq t_n$ and $\min\Gamma(t_n)\le
                    s_n\preceq s_i$. From \Cref{lemma:ell_gamma_convex} it
                    follows that $t_i\in \Gamma(s_n)$.
            \end{itemize}
            In any case, $(t_i,s_i)\in W_i$ satisfies the inductive hypothesis,
            and we are done by induction on the elements of $W_{i-1}$
            (which is empty for $i=1$).

        \item If $t_i s_i$ is not $r$-dependent on any edge in $I_{i+1}$: by
            \Cref{alg:independent}, we have $t_i s_i\in I_i$ and
            $I_{i+1}=I_i\setminus \{ t_i s_i \}$.

            By definition, $W_{i-1}$ does not cover $t_i s_i$.

            If $\exists (h,v)\in W_{i-1}$ such that $h\in \Gamma(s_i)$
            and $h\prec t_i$, then $W_{i-1}\setminus W_i$ cannot be
            empty (see \cref{line:ifextend} in \Cref{alg:main}), which
            contradicts the initial assumption of this case.

            It follows that any $(h,v)\in W_{i}\setminus \{(t_i,s_i)\}$
            satisfies the inductive hypothesis for $W_i$ by the inductive
            hypothesis on $W_{i-1}$ (which is empty for $i=1$).
            By~\eqref{eq:Idecreasing}, we have $t_i s_i\in I_{i}\subseteq I_1$.
            All in all, every element of $W_i$ satisfies the inductive
            hypothesis.
    \end{itemize}

	\paragraph{Case 3: \boldmath$W_i=W_{i-1}-(t_m,s_j)+(t_i,s_j)$ and $t_i s_i$
	is $r$-independent from every edge in $I_{i+1}$.} By the explicit design of
	\Cref{alg:independent}, we are (almost trivially) done by induction.
	\begin{itemize}
        \item If $t_i s_i\in I_i$: from~\eqref{eq:Idecreasing}, we get $t_i
            s_i\in I_i\subseteq I_1$. We are done by induction.

		\item If $t_i s_i\notin I_i$: by \Cref{alg:independent}, we have
			$I_{i+1}=I_i$. Furthermore, $(t_m,s_j)$ does not cover any element
			of $I_{i+1}$, and $\nexists t_n s_n\in I_{i+1}$ such that $t_m\in
			\Gamma(s_n)$ and $t_m\prec t_n$. Thus, by induction on
			$(t_m,s_j)\in W_{i-1}$, there $\exists q\in\mathbb{N}$ such that
			$1\le q\le i-1$, $t_q s_q\in I_1$ and $(t_q,s_j)\in W_q$. This $q$
			is also a good choice for $(t_i,s_j)\in W_i$, and we are done by
			induction on the rest of the elements of $W_i$.
	\end{itemize}

	\paragraph{Case 4: \boldmath$W_i=W_{i-1}-(t_m,s_j)+(t_i,s_j)$ and $\exists
	t_p s_p\in I_{i+1}$ that is $r$-dependent on $t_i s_i$.} By
	\Cref{alg:independent}, we have $I_{i+1}=I_i$. By \Cref{lemma:ti_distinct},
	we have $j\le m<i$. Note, that by \Cref{lemma:alg_order}, we have $t_m\in
	\Gamma(s_i)$, $t_m\prec t_i$, $\min\Gamma(t_i)\in \Gamma(t_m)$, and
	$\min\Gamma(s_j)\in \Gamma(s_i)$. By \Cref{lemma:alg_semiguards}, we have
    $\min\Gamma(s_j)\preceq t_i$ and $\min\Gamma(s_j)\in \Gamma(s_i)$. By
    \Cref{lemma:Ii_elements}, we have $p>i$, therefore $s_p\prec s_i$. First,
    we deal with two cases where using the inductive hypothesis is easy.
	\begin{itemize}
		\item If $t_i\prec t_p$: then by $r$-dependence,
			$\min\Gamma(s_i)\in \Gamma(s_p)$ and
			$\min\Gamma(t_p)\in \Gamma(t_i)$. Therefore
			\Cref{lemma:ell_gamma_convex} applies to
			$\min\Gamma(s_i)\le t_i\preceq t_p$ and
			$\min\Gamma(t_p)\le s_p\preceq s_i$, so $t_i\in
			\Gamma(s_p)$. It follows that $(t_i,s_j)\in W_i$
			satisfies the inductive hypothesis, and
			the rest of the statement follows by induction.

		\item If $\min\Gamma(s_j)\preceq t_p\preceq t_i$: by $r$-dependence,
			$\min\Gamma(t_i)\in \Gamma(t_p)$. Note, that
			$\min\Gamma(s_i)\le\min\Gamma(s_j)\preceq t_m$.
			\begin{itemize}
				\item If $\min\Gamma(t_i)\preceq s_p$: we have
					$\min\Gamma(s_i)\preceq \min\Gamma(s_j)\preceq t_p$ and
					$\min\Gamma(t_i)\preceq s_p\preceq s_i$.
					By \Cref{lemma:ell_gamma_convex}, we have
					$\min\Gamma(s_j)\in \Gamma(s_p)$.

				\item If $s_p\preceq \min\Gamma(t_i)$: note,
					that $\min\Gamma(s_j)\preceq t_p$ and
					$s_p\preceq \min\Gamma(t_i)$, so by
					\Cref{lemma:inheriting_neighbor}, we have
					$\min\Gamma(s_j)\in \Gamma(s_p)$.
			\end{itemize}
			In any case, $(t_i,s_j)$ covers
			$t_p s_p \in I_{i+1}$, and we are done by induction.
	\end{itemize}

	Suppose from now on, that $t_p\prec \min\Gamma(s_j)$. By
	$r$-dependence, $\min\Gamma(s_i)\in \Gamma(s_p)$, $\min\Gamma(t_i)\in
    \Gamma(t_p)$. As $\min\Gamma(s_j)\preceq t_m$, we have
	$t_p\prec t_m\prec t_i$.

    \medskip

    By applying \Cref{lemma:inheriting_neighbor} to $t_p\prec t_m$ and
    $\min\Gamma(t_m)\le \min\Gamma(t_i)$, we conclude that
    \begin{equation*}
        \min\Gamma(t_m)\in \Gamma(t_p).
    \end{equation*}
	We have three more sub-cases. Recall that $I_{i+1}=I_i$ in
    this case. Note, that $(t_m,s_j)\in W_{i-1}$
	satisfies the inductive hypothesis.
	\begin{itemize}
		\item If $\exists t_n s_n\in I_{i}$ such that
			$(t_m,s_j)$ covers $t_n s_n$: then
			$\min\Gamma(t_m)\in \Gamma(t_n)$ and
			$\min\Gamma(s_j)\in \Gamma(s_n)$.
			Note, that $\min\Gamma(s_i)\le \min\Gamma(s_j)$
			and $s_n\prec s_i$. By
			\Cref{lemma:inheriting_neighbor}, we have
			$\min\Gamma(s_i)\in \Gamma(s_n)$.

			It follows that $\min\Gamma(s_i)\in
			\Gamma(s_n)\cap\Gamma(s_p)$ and $\min\Gamma(t_m)\in
			\Gamma(t_n)\cap\Gamma(t_p)$. By
			\Cref{lemma:Ii_indep}, $I_{i+1}$ is a
			set of pairwise $r$-independent edges, so we must have
			$p=n$. Therefore $(t_m,s_j)$ covers $t_p s_p$, which
			implies that $\min\Gamma(s_j)\in \Gamma(s_p)$. It
			follows that $(t_i,s_j)$ also covers $t_p s_p$,
			therefore we are done by induction.

		\item If $\exists t_n s_n\in I_{i}$ such that
			$t_m\in \Gamma(s_n)$ and $t_m\prec t_n$:
            note, that \Cref{lemma:inheriting_neighbor} applies to $t_m\prec
            t_n$ and $\min\Gamma(t_n)\le s_n$, therefore
			$\min\Gamma(t_n)\in \Gamma(t_m)$.
            \Cref{lemma:inheriting_neighbor} also applies to
            $s_n\preceq s_i$ and $\min\Gamma(s_i)\le t_m$,
			thus
            \begin{equation*}
                \min\Gamma(s_i)\in \Gamma(s_n).
            \end{equation*}
            We have three sub-cases.
			\begin{itemize}
				\item If $\min\Gamma(t_n)\preceq \min\Gamma(t_i)$:
					note that $t_p\preceq t_m$, so by
					\Cref{lemma:inheriting_neighbor}, we have
					$\min\Gamma(t_n)\in \Gamma(t_p)$.
					Thus $t_n s_n$ and $t_p s_p$ are $r$-dependent, so we
					must have $p=n$ by \Cref{lemma:Ii_indep}. This is a contradiction
					because $t_p\preceq t_m\prec t_n$.

                \item If $\min\Gamma(t_i)\prec \min\Gamma(t_n)$ and $s_n\le
                    s_i$: then $\min\Gamma(t_n)\le s_n\le s_i$ and
                    $\min\Gamma(t_i)\le s_i$. By \Cref{lemma:le_comparable},
                    $\min\Gamma(t_i)$ and $\min\Gamma(t_n)$ are comparable, so
                    we must have $\min\Gamma(t_i)<\min\Gamma(t_n)$. This implies
                    that $t_i\prec t_n$ and
					\begin{equation*}
                        s_n\in [\min\Gamma(t_n), s_i]\subseteq [\min\Gamma(t_i),
                        s_i]\subseteq \Gamma(t_i),
					\end{equation*}
                    i.e., $t_i\in \Gamma(s_n)$. Therefore $(t_i,s_j)$ satisfies
                    the inductive hypothesis for $I_{i+1}$ and we are done by
                    induction.

                \item If $\min\Gamma(t_i)\prec \min\Gamma(t_n)$ and $s_n$ is
                    not comparable to $s_i$: note, that $s_i,s_n\in
                    \Gamma(t_m)$. Observe, that \Cref{lemma:inheriting_neighbor}
                    applies to $\min\Gamma(s_i)\le t_m$ and $s_n\preceq s_m$,
                    thus
                    \begin{equation*}
                        s_i,s_n\in \Gamma(\min\Gamma(s_i)).
                    \end{equation*}

                    Note, that $\min\Gamma(s_i)\le t_m$ and $\min\Gamma(t_n)\le
                    s_n$, and \Cref{lemma:inheriting_neighbor} applies,
                    therefore
                    \begin{equation*}
                        \min\Gamma(t_n)\in \Gamma(\min\Gamma(s_i)).
                    \end{equation*}
                    Similarly, since $\min\Gamma(t_i)\le s_i$,
                    \begin{equation*}
                        \min\Gamma(t_i)\in \Gamma(\min\Gamma(s_i)).
                    \end{equation*}
					It follows that $\min\Gamma(t_n)\in [\min\Gamma(t_i),s_n]$.
					We also have $s_p\in \Gamma(\min\Gamma(s_i))$.

					Note, that $s_i$ and $s_n$ cut
					$\Gamma(\min\Gamma(s_i))$ into three parts.
					Since $s_p\prec s_i$, either $s_n<s_p$
					or $s_p\le s_n$ or $s_p\le s_i$ holds.

                    \medskip

					If $s_n<s_p$, then $\min\Gamma(t_n),\min\Gamma(t_p)\le s_p$
                    are comparable by \Cref{lemma:le_comparable}, so
					\begin{equation*}
						\min\Gamma(t_n)\in
						[\min\Gamma(t_i),s_n]\subseteq
						[\min\Gamma(t_i),s_p]\subseteq
						\Gamma(t_p).
					\end{equation*}
					Thus if $s_n<s_p$, then $t_n s_n$ and $t_p s_p$ are $r$-dependent, so we
					must have $p=n$ by \Cref{lemma:Ii_indep}. However, $p=n$
                    contradicts $t_p\preceq t_m\prec t_n$.

                    \medskip

                    It follows that $s_p\le s_i$ or $s_p\le s_n$ holds. Since
                    $s_i$ is not comparable to $s_n$, and $s_i,s_n,s_p\in
                    \Gamma(\min\Gamma(s_i))$, it follows that
                    \begin{equation*}
                        s_p\in [s_i,s_n]\subseteq \Gamma(\min\Gamma(s_i)).
                    \end{equation*}
					We have $s_i,s_n\in \Gamma(t_m)$. Note,
					that $\min\Gamma(s_j)\preceq t_m$ and
					$s_n,s_i\preceq s_m$. By
					\Cref{lemma:alg_semiguards},
					$\min\Gamma(s_j),t_m\in \Gamma(s_m)$.
                    By \Cref{lemma:inheriting_neighbor}, it follows that
                    $s_n,s_i\in \Gamma(\min\Gamma(s_j))$. Then
					\begin{equation*}
						s_p\in [s_i,s_n]\subseteq
						\Gamma(\min\Gamma(s_j)).
					\end{equation*}
                    Thus we have shown that $(t_i,s_j)$ covers $t_p s_p$, and we
                    are done by induction.
			\end{itemize}

        \item If $\nexists t_n s_n\in I_{i}$ such that $(t_m,s_j)$ covers $t_n
            s_n$, and $\nexists t_n s_n\in I_{i}$ such that $t_m\in \Gamma(s_n)$
            and $t_m\prec t_n$: by induction, there $\exists q\in\mathbb{N}$
            such that $1\le q\le i-1$, $t_q s_q\in I_1$ and $(t_q,s_j)\in W_q$.
            Since $q$ is a good choice for $(t_i,s_j)\in W_i$ as well, the
            inductive hypothesis holds for $W_i$.
	\end{itemize}
\end{proof}
\end{document}